\documentclass[11pt]{amsart}
\usepackage{amsfonts, bm}
\usepackage{mathrsfs}
\allowdisplaybreaks
\usepackage{hyperref}

\makeatletter

\renewcommand{\tocsection}[3]{%
  \indentlabel{\@ifnotempty{#2}{\bfseries\ignorespaces#1 #2\quad}}\bfseries#3}

\usepackage{amsmath}
\usepackage{amssymb}
\usepackage{multicol}
\usepackage{stmaryrd}
\usepackage{cite}
\usepackage{epsfig}
\usepackage{color}
\usepackage{graphics}
\usepackage{graphicx}
\usepackage{multicol,graphics}
\newcommand\bes{\begin{eqnarray}}
\newcommand\ees{\end{eqnarray}}
\newcommand\R{\mathbb R}
\newtheorem{theorem}{Theorem}[section]
\newtheorem{lemma}[theorem]{Lemma}

\newtheorem{remark}[theorem]{Remark}

\numberwithin{equation}{section}

\theoremstyle{plain}
\newtheorem*{theorem*}{Theorem A}

\newcommand\bess{\begin{eqnarray*}}
	\newcommand\eess{\end{eqnarray*}}
\newcommand{\lf}{\left}
\newcommand{\rr}{\right}
\newcommand{\dd}{\displaystyle}

\newcommand{\td}{\tilde}

\newcommand\yy{\infty}

\newcommand{\ol}{\overline}
\newcommand{\rd}{{\rm d}}

\begin{document}

\title[Dynamics of a competition model]{Long-time dynamics of a competition model with\\  nonlocal diffusion and free boundaries: \\ Vanishing and spreading of the invader}
\author[Y. Du, W. Ni, L. Shi ]{Yihong Du$^\dag$,  Wenjie Ni$^\dag$ and Linfei Shi$^{\ddag}$}
\thanks{\hspace{-.5cm}
$^\dag$ School of Science and Technology, University of New England, Armidale, NSW 2351, Australia.
\\
$^{\ddag}$ School of Mathematics and Statistics, Beijing Institute of Technology, Beijing 100081, China.
\\
\mbox{\ \  Emails:} ydu@une.edu.au (Y. Du),\ wni2@une.edu.au (W. Ni),\ 3120205702@bit.edu.cn (L. Shi)}

\date{\today}

\begin{abstract}

In this work,   we investigate the long-time dynamics of a two species competition model of Lotka-Volterra type with nonlocal diffusions. One of the species, with density $v(t,x)$,  is assumed to be a native  in the environment (represented by the real line $\R$), while the other species, with density $u(t,x)$, is an invading species which invades the territory of $v$ with two fronts, $x=g(t)$ on the left and $x=h(t)$ on the right. So the population range of $u$ is the evolving interval $[g(t), h(t)]$ and the reaction-diffusion equation for $u$ has two free boundaries, with $g(t)$  decreasing in $t$ and $h(t)$ increasing in $t$, and the limits $h_\infty:=h(\infty)\leq \infty$ and $g_\infty:=g(\infty)\geq -\infty$ thus always exist.
 We obtain detailed descriptions of the long-time dynamics of the model according to whether $h_\infty-g_\infty$ is $\infty$ or finite.
In the latter case, we reveal in what sense the invader $u$ vanishes in the long run and $v$ survives the invasion, while in the former case, we obtain a rather satisfactory description of the long-time asymptotic limit for both $u(t,x)$ and $v(t,x)$ when a certain parameter $k$ in the model is less than 1. This research is continued in a separate work, where sharp criteria are obtained to distinguish the case $h_\infty-g_\infty=\infty$ from the case $h_\infty-g_\infty$ is  finite, and new phenomena are revealed for the case $k\geq 1$. The techniques developed in this paper should have applications to other models with nonlocal diffusion and free boundaries.

\bigskip

\medskip


\end{abstract}

\maketitle

\section{Introduction}
We are interested in the long-time dynamics of the following Lotka-Volterra type competition model with nonlocal diffusion and free boundaries
\begin{align}\label{KnK1.2}
	\begin{cases}
	\dd	u_t = d_1\int_{g(t)}^{h(t)}J_1(x - y)u(t, y)\rd y - d_1u + u(1 - u - kv),   &
		t > 0, ~g(t) < x < h(t),\\[4mm]
\dd		v_t = d_2\int_\mathbb{R}J_2(x - y)v(t, y)\rd y - d_2v + \gamma v(1 - v - hu),   &
		t > 0,~x\in\mathbb{R},\\[3mm]
		u(t, x) = 0,  &
		t \geq 0,\ x \not\in (g(t), h(t)),
		\\[2mm]
	\dd	h^{\prime} (t) = \mu\int_{g(t)}^{h(t)}\int_{h(t)}^{\infty}J_1(x - y)u(t, x)\rd y\rd x,  &
		t > 0,\\[4mm]
	\dd	g^{\prime} (t) = -{\mu}\int_{g(t)}^{h(t)}\int_{-\infty}^{g(t)}J_1(x - y)u(t, x)\rd y\rd x,  &
		t > 0,\\[3mm]	
		h(0) = -g(0) = h_0 > 0,\ u(0, x) = u_0(x), \
		 v(0, x) = v_0(x),
		&
		x \in\mathbb{R},
	\end{cases}
\end{align}
where $d_1, d_2, h, k, \gamma, \mu$ are given positive constants, and the initial functions satisfy 
\begin{align}\label{KnK1.3}
		u_{0} \in C(\R), ~u_{0}(x) = 0 ~ {\rm for}~|x|\geq h_0, ~ u_{0}(x) > 0  ~ {\rm for}~|x|< h_0,\ 
		v_0\in C_b(\mathbb{R}), ~ v_{0}(x) \geq 0 {\rm ~in~} \mathbb{R},
\end{align}
where $C_b(\mathbb{R})$ is the space of continuous and bounded functions in $\mathbb{R}$. The kernel functions $J_1, J_2$ are assumed to satisfy\smallskip

\noindent$(\mathbf{J}):$ \ \ \ $J_i\in C_b(\mathbb{R})$, $J_i(x)=J_i(-x) \geq 0,$ $J_i(0) > 0,$ $\dd\int_{\mathbb{R}}J_i(x)dx = 1$ for $i=1,2$.

\medskip

System \eqref{KnK1.2} may be viewed as a model describing the invasion of  a species with density $u$ into an environment (represented by $\R$ here) where a native competitor, with density $v$, has already appeared.  The population range of $u$ is given by the time-dependent interval $[g (t), h(t)]$, with $x=g (t)$ and $x=h(t)$ known as the free boundaries in the model, which represent the range boundary of $u$, or its invading fronts. 

The corresponding local diffusion version of \eqref{KnK1.2} was first studied in \cite{DuLin}.
As will be explained below, the nonlocal diffusion model \eqref{KnK1.2} poses several technical difficulties in the mathematical treatment, and is capable of exhibiting strikingly different behaviour from its local diffusion correspondent.

We remark that \eqref{KnK1.2}  is a reduced version of the following equivalent but more general looking system:
\begin{align}\label{KnK1.1}
	\begin{cases}
	\dd	U_t = D_1\int_{G(t)}^{H(t)}J_1(x - y)U(t, y)\rd y - D_1U + U(a_1 - b_1U - c_1V),   &
		t > 0, ~G(t)	< x < H(t),\\
	\dd	V_t = D_2\int_\mathbb{R}J_2(x - y)V(t, y)\rd y - D_2V + V(a_2 - b_2V - c_2U),   &
		t > 0,~x\in\mathbb{R},\\
			U(t, x) = 0,  &
		t \geq 0,x \not\in (G(t), H(t)),\\
\dd		H^{\prime} (t) = \hat{\mu}\int_{G(t)}^{H(t)}\int_{H(t)}^{\infty}J_1(x - y)U(t, x)\rd y\rd x,  &
		t > 0,\\
\dd		G^{\prime} (t) = -\hat{\mu}\int_{G(t)}^{H(t)}\int_{-\infty}^{G(t)}J_1(x - y)U(t, x)\rd y\rd x,  &
		t > 0,\\	
		H(0) = -G(0) = H_0 > 0, U(0, x) = U_0(x), &
		-H_0 \leq x \leq H_0,\\
		V(0, x) = V_0(x),
		&
		t \geq 0,x \in\mathbb{R},
	\end{cases}
\end{align}
Indeed, let
$$u(t, x) = \frac{b_1}{a_1}U\left(\frac{t}{a_1}, x\right), ~v(t, x) = \frac{b_2}{a_2}V\left(\frac{t}{a_1}, x\right), ~h(t) = H\left(\frac{t}{a_1}\right), ~g(t) = G\left(\frac{t}{a_1}\right),$$
$$d_1 = \frac{D_1}{a_1}, ~d_2 = \frac{D_2}{a_1}, ~\gamma = \frac{a_2}{a_1}, ~k = \frac{a_2c_1}{a_1b_2}, ~h = \frac{a_1c_2}{a_2b_1}, ~\mu=\frac{1}{b_1}\hat{\mu}.$$
Then a simple calculation shows that (\ref{KnK1.1}) is reduced to  \eqref{KnK1.2}.
\medskip

Several closely related models have been studied recently. In \cite{DY-2022-DCDS}, the situation that
 the population range of $v$ is the same evolving interval $[g(t), h(t)]$ was considered, and a technical difficulty in treating the nonlocal competition model has been revealed: Due to the lack of compactness of the solutions of the nonlocal diffusion problem, the fact $[g(t), h(t)]$ remains uniformly bounded for all time $t>0$, does not lead to the conclusion of vanishing as in the corresponding local diffusion models. To recover the vanishing property, \cite{DY-2022-DCDS} relied on the following extra assumption
 \begin{equation}\label{J>0}
 J_i(x)>0 \mbox{ for all } x\in\R,\ i=1,2,
 \end{equation}
 and  a trick relating the left and right derivatives of
 $
 M(t):=\max_{x\in [g(t), h(t)]}u(t,x)$ to $ u_t(t,\xi(t))$  for some suitable $ \xi(t)\in [g(t), h(t)]$.
 It was conjectured in \cite{DY-2022-DCDS} that the assumption \eqref{J>0} is unnecessary. 
 
 In \cite{CLWZ}, the same system \eqref{KnK1.2} was considered, while \cite{Zwy-2022-dcds-b} considered the special case $d_1=d_2$ and $J_1=J_2$. In \cite{LLW}, a predator-prey system with both species evolving over the same interval $[0, h(t)]$ was investigated. In these papers, the existence and uniqueness of the solution was established, together with various results on the long-time dynamics of the model for certain selected ranges of the parameters. However, to prove vanishing of $u$, they used the same trick as in \cite{DY-2022-DCDS} and made the same additional assumption \eqref{J>0} for $J_1(x)$.
 
 In this work, we will introduce new techniques to treat \eqref{KnK1.2} without requiring \eqref{J>0}, and obtain  precise description of the long-time dynamics of the model for broader parameter ranges.  Moreover, we will reveal
 several new behaviours of the model.

\medskip

Let us start by recalling the following well-posedness result:
\begin{theorem*}{\rm (\!\! \cite{CLWZ})}
Assume   $(\mathbf{J})$ holds,  and  the initial functions satisfy \eqref{KnK1.3}. Then \eqref{KnK1.2} admits a unique solution $(u, v, g, h)$ defined for all $t > 0.$
\end{theorem*}

For the long-time dynamics, we will use some terminologies arising from  the corresponding ODE version of \eqref{KnK1.2}, namely
\begin{equation}\label{ODE}
 \begin{cases}
 u'=u(1-u-kv),\\ v'=\gamma v(1-v-hu), \\
u(0)>0,\ v(0)>0,
\end{cases}
\end{equation}
which always has
 the trivial equilibrium $R_0=(0,0)$ and semi-trivial equilibria  $R_1=(1, 0)$ and $R_2=(0,1)$.
Moreover, if $\min\{h, k\}>1$ or $\max\{h, k\}<1$, the problem has a unique positive equilibrium
$R_*:=(\frac{1-k}{1-hk}, \frac{1-h}{1-hk})$. Regarding the long-time dynamics of \eqref{ODE}, the following conclusions are well known:
\begin{enumerate}
\item $R_0$ is always unstable;
\item  when $\max\{h, k\}<1$, $R_*$ is globally asymptotically stable;
\item when $k<1<h$, $R_1$ is globally asymptotically stable;
\item when $h<1<k$, $R_2$ is globally asymptotically stable;
\item when $\min\{h, k\}>1$, both $R_1$ and $R_2$ are locally asymptotically stable while $R_*$ is unstable.
\end{enumerate}

In case (2), the competitors co-exist in the long run, and it is often referred to as the {\it weak competition} case, where no competitor wins or loses in the competition.
Cases (3) and (4) are known as the {\it weak-strong competition} cases. In case (3), the competitor $u$ wipes $v$ out in the long run and wins the competition; so we will call $u$ the strong competitor
and $v$ the weak competitor. Analogously $u$ is the weak competitor and $v$ is the strong competitor in case (4).
Case (5) is the {\it strong competition} case, and depending on the location of $(u(0), v(0))$ in the first quadrant of $\R^2$, one of $R_1$ and $R_2$ is the global attractor of \eqref{ODE} (except when $(u(0), v(0))$ lies on the one dimensional  stable manifold of $R_*$).
We will keep using these terminologies which are determined solely by $h$ and $k$ for \eqref{KnK1.2}.

To state our main results in this paper, we also need to recall some properties of  the principle eigenvalue for some associated nonlocal diffusion operators. For $\Omega=(a,b)$ a finite interval, under our assumption ({\bf J}),  it is well known  that  the
 following eigenvalue problem
\begin{equation}\label{eigen}
	\begin{aligned}
\lambda\varphi=		\mathcal{L}_{\Omega} [\varphi](x): = d_1\left[\int_{\Omega}J_1(x - y)\varphi(y)\rd y - \varphi(x)\right] , ~\varphi\in C(\overline{\Omega}), 
		\nonumber
	\end{aligned}
\end{equation}
 has a principal eigenvalue 
$\lambda=\lambda_p( \mathcal{L}_{\Omega})$ associated with a positive eigenfunction $\varphi$ (e.g.,\cite{BCV, C, LCW}), and it has the following properties:\medskip

\noindent {\bf Proposition B.} (\!\!\cite[Proposition 3.4]{Cao-2019-JFA}){\it \
Assume that $l > 0$, and $J_1$ satisfies $(\mathbf{J})$.  Then 
	
	{\rm (i)} $\lambda_p(\mathcal{L}_{(a, a+l)})=\lambda_p(\mathcal{L}_{(0, l)})$  for all  $ a\in\R$,
	
	{\rm(ii)} $\lambda_p(\mathcal{L}_{(0, l)})$ is  strictly increasing and continuous in $l$,
	
	{\rm(iii)} $\lim\limits_{l \rightarrow \infty}\lambda_p(\mathcal{L}_{(0, l)}) = 0$,
	
	{\rm(iv)} $\lim\limits_{l \rightarrow 0}\lambda_p(\mathcal{L}_{(0, l)}) =- d_1$.}

\medskip

We are now ready to state our main results in this paper.  Let $(u,v,g,h)$ be the solution of \eqref{KnK1.2}. Then  $g_\yy:=\lim_{t\to\yy}g (t)\in [-\infty, -h_0)$ and $h_\yy:=\lim_{t\to\yy}h(t)\in (h_0, \infty]$ always exist. We will  describe the long-time dynamics of \eqref{KnK1.2} according to the following two cases:
\[
{\bf (a)\!:} \  h_\infty-g_\infty<\infty,\ \ {\bf (b)\!:} \ h_\infty-g_\infty=\infty.
\]
In case {\bf (a)} the limit of the population range of $u$ is finite and one expects $u$ to vanish in the long run, while in case  {\bf (b)}, the limit of the size of the population range of $u$ is infinite, and successful invasion of $u$ is expected. 

For case {\bf (a)}, by introducing new techniques for the analysis of the function 
\[
U(t,x):=\int_{x-L}^{x+L}u(t,y)\rd y\]
 for suitable values of $L$, 
 we will prove the following result.

\begin{theorem}\label{th1.1} Assume that $(\mathbf{J})$  holds and $(u, v, g, h)$ is the unique solution of \eqref{KnK1.2}.  If  $h_\infty - g_\infty < \infty$, then necessarily
	\begin{align}\label{3.3}
	\mbox{ $d_1> 1-k$  and } \lambda_p(\mathcal{L}_{(g_\infty, h_\infty)})\leq  k- 1;
	\end{align}
	moreover
	\begin{equation}\label{vanishing}\begin{cases}
	\dd\lim_{t\to\infty}\int_{\R}u(t,x)dx=0,\\[2mm]
	\dd\lim_{t\to\infty}\int_L^L|v(t,x)-1)|dx=0 \mbox{ for every } L>0,\\[2mm]
	\dd\lim_{t\to\infty} v(t,x)=1 \mbox{ locally uniformly for } x\in\R\setminus (g_\infty, h_\infty).
	\end{cases}
	\end{equation}
\end{theorem}

One naturally wonders whether in Theorem \ref{th1.1} the conclusions in \eqref{vanishing} imply the following stronger statements, as  in the corresponding local diffusion model \cite{DuLin}: 
\begin{equation}\label{3.4}
		\begin{aligned}
			&\lim\limits_{t \rightarrow \infty}\max\limits_{x\in[g(t), h(t)]} u(t, x) = 0\ \mbox{ and } \  \lim\limits_{t \rightarrow \infty}v(t, x) = 1 {\rm ~locally ~uniformly ~ for~} x\in\mathbb{R}.
		\end{aligned}
	\end{equation}
It turns out that this question is not easy to answer. Using \eqref{vanishing} and viewing \eqref{KnK1.2}  as a perturbation of the corresponding ODE problem \eqref{ODE}, we are able to obtain a partial answer to this question. 

Let us now be more precise.
 Denote $\td d_1:=d_1+k-1$ and 
\begin{align}\label{F}
	F(s):=\gamma (1-hk)s^2+[\td d_1\gamma h-\gamma(1-hk)-d_2]s-\td d_1\gamma h,
\end{align}
which arises from the analysis of \eqref{ODE}.
Define the sets	$\Theta_1$ and $\Theta_2$  by
\begin{equation}\label{theta}
	\begin{cases}
		\Theta_1:=\{(\gamma,h,k,d_1,d_2)\in \R^5_+: \ F(s)\not= 0 \ {\rm for }\  s\in [0,1] \},\\
		\Theta_2:=\{(\gamma,h,k,d_1,d_2)\in\R^5_+: \ F(s)=0 \ {\rm has \ at\ least\ one\ root\ in}\  [0,1] \}.
	\end{cases}
\end{equation}
Clearly $\Theta_2=\R^5_+\setminus \Theta_1$, so for any given parameters $(\gamma,h,k,d_1,d_2)$ in \eqref{KnK1.2}, it belongs to either $\Theta_1$ or $\Theta_2$.
It is easy to show  (see Remark \ref{remak3} below) that $(\gamma,h,k,d_1,d_2)\in \Theta_1$ if
\begin{align*}
	d_1\geq 1 \  {\rm or}\ \dd kh\leq 1+\frac{d_2}{\gamma}.
\end{align*}
(So in the weak competition case where $k, h\in (0,1)$, we always have  $(\gamma,h,k,d_1,d_2)\in \Theta_1$. )
\medskip
	
Regarding the above question on the validity of \eqref{3.4}, we have the following answer.

\begin{theorem}\label{th1.2} Under the assumptions of Theorem \ref{th1.1},
	 the following conclusions hold:
	\begin{itemize}
	\item[{\rm (i)}] If $d_1\geq 1$ or if $d_1<1$ and $(\gamma,h,k,d_1,d_2)\in\Theta_1$, then
	\eqref{3.4} holds.
	\item[{\rm (ii)}] If $d_1<1$ and $(\gamma,h,k,d_1,d_2)\in\Theta_2$, then either \eqref{3.4} still holds or there is an open set\ $\Omega\subset (g_\yy,h_\yy)$  with $	|\Omega|=h_\yy-g_\yy$, $\Omega\not=(g_\infty, h_\infty)$, such that 
	\begin{equation}\label{Omega}
		\lim_{t\to \yy} (u(t,x),v(t,x))=\begin{cases} (0,1)\ \ &\mbox{locally uniformly for } x\in \Omega,\\[2mm]
		(kx_*-\td d_1,1-x_*)\ \ &\mbox{ for }  x\in  (g_\yy,h_\yy)\backslash \Omega,
		\end{cases}
	\end{equation}
	where $x_*\in (0, 1)$ is the smallest positive root of $F(s)=0$ in $[0, 1]$.
\end{itemize}
\end{theorem}

\begin{remark}
We do not know whether \eqref{Omega} can actually happen, and conjecture that it never happens.
\end{remark}

For the case  $h_\yy-g_\yy=\yy$, we will prove the following result.

\begin{theorem}\label{th1.4}
Assume that $(\mathbf{J})$  holds and $(u, v, g, h)$ is the unique solution of \eqref{KnK1.2}.
	If  $h_\yy-g_\yy=\yy$ and $k<1$, then   $h_\infty = \infty,$ $g_\infty = -\infty$
	and
	\[
	\lim_{t\to\infty}(u(t,x), v(t,x))=\begin{cases} (1,0) &\mbox{ if } h\geq 1,\\
	(\frac{1-k}{1-hk},\frac{1-h}{1-hk}) &\mbox{ if } h<1,
	\end{cases}
	\]
	where the convergence is locally uniform for $x\in\R$.
\end{theorem}

The assumption $k<1$ in Theorem \ref{th1.4} cannot be removed. In a separate following up work \cite{dns-part2}, we will show that, when $k\geq 1$, it is possible to have $h_\infty=\infty$ while $g_\infty$ is finite.  Since \eqref{KnK1.2} generates a monotone dynamical system, it is possible to obtain sharp criteria for $h_\infty-g_\infty=\infty$ or $h_\infty-g_\infty<\infty$ to happen; this will also be carried out in this following up paper.

For nonlocal diffusion models with free boundary, compared with the corresponding models with local diffusion, a new phenomena occurs on the asymptotic spreading speed, namely accelerated spreading may happen. For the Fisher-KPP single species model, and for certain cooperative models, such a phenomena was investigated recently in \cite{DLZ-2021-JMPA, dn-jde, dn-jems, dn-ma, dnw-non}. We leave the study of this behaviour of the competition system here to future work.

There are extensive recent works on competition systems with nonlocal diffusion over a fixed bounded domain or over the entire Euclidean space. Since no free boundaries appear
in these situations, compared to our work here, usually significantly different techniques are used; indeed, the technical difficulties here do not appear in these problems.
For works on a bounded domain, we mention as examples  \cite{Lf-jmaa-2014, Bxl-jde-2015, Bxl-cvpde-2018, Bxl-dcds-2020} and the references therein.
For works on nonlocal competition systems over the entire space, an incomplete sample includes \cite{Swx-cpaa-2012, Lwt-dcds-2015, Zgb-cvpde-2020,   Wcf-jde-2019, Qsx-jde-2024} and the references therein. In \cite{Swx-cpaa-2012}, the authors obtained conditions for coexistence and extinction of the  species, and also considered the system in higher dimensions. 
In \cite{Lwt-dcds-2015}, the authors  constructed entire solutions behaving like  two monotone traveling-wave solutions moving toward each other from the ancient time $-\infty$,
giving rise to the phenomena of successful invasion of a species with speed $c_1$ from $x=-\infty$ and with speed $c_2$ from $x=+\infty$. 
In  \cite{Wcf-jde-2019} and \cite{Qsx-jde-2024},   the propagation behaviour under a shifting environment was investigated.
In \cite{Zgb-cvpde-2020},  a strong competition system was considered, where interesting results on the stability of bistable traveling waves and the long-time propagation behaviour of the system were obtained.

 The rest of this paper is organised as follows. In Sections 2 and 3, we prove Theorems \ref{th1.1} and \ref{th1.2}, respectively. The proof of Theorem \ref{th1.4} is given in Section 4.

\section{Proof of Theorem \ref{th1.1}}  
We prove Theorem \ref{th1.1} 
by a sequence of lemmas.

\begin{lemma}\label{lem:2.2}
If  $h_\infty - g_\infty < \infty$, then \eqref{3.3} holds. 
\end{lemma}
\begin{proof}
	To prove the second inequality in \eqref{3.3}, we argue indirectly: Assume, on the contrary, that $\lambda_p(\mathcal{L}_{(g_\infty, h_\infty)})>k-1.$ Then $\lambda_p(\mathcal{L}_{(g_\infty, h_\infty)}>  k(1 + \epsilon) -1$ for sufficiently small $\epsilon > 0,$ say $\epsilon\in(0, \epsilon_1).$ For such $\epsilon,$ we can find $T_\epsilon > 0$ such that 
	\begin{align*}
		&h(t) > h_\infty - \epsilon, ~g(t) < g_\infty + \epsilon \ \ {\rm ~for~} t > T_\epsilon,\\
		&v(t, x) < 1 + \epsilon\ \ {\rm ~for~} t > T_\epsilon,\ x\in \R.
	\end{align*}
	Let $w_\epsilon$ be the unique solution of the auxiliary problem
	\begin{align}
		\begin{cases}
			\dd	{w}_t = d_1\int_{g_\infty  + \epsilon}^{h_\infty - \epsilon}J_1(x - y)w(t, y)\rd y - d_1w + w(1 - w -k(1 + \epsilon)), &
			t > T_\epsilon, ~x\in[g_\infty + \epsilon, h_\infty - \epsilon],\\ 
			w(T_\epsilon, x) = u(T_\epsilon, x), & x\in[g_\infty + \epsilon, h_\infty - \epsilon].
			\nonumber
		\end{cases}
	\end{align}
	Since $\lambda_p(\mathcal{L}_{(g_\infty, h_\infty)} )> k(1 + \epsilon) -1$, from   \cite[Proposition 3.5]{Cao-2019-JFA}, it follows that $w_\epsilon$ converges to the unique positive steady-state $W_\epsilon(x)$ uniformly in $[g_\infty + \epsilon, h_\infty - \epsilon]$ as $t \rightarrow \infty.$ Moreover, a simple comparison argument yields that $u(t, x) \geq w_\epsilon(t, x)$ for $t > T_\epsilon,$ and $x\in [g_\infty + \epsilon, h_\infty - \epsilon].$ Therefore, there exists $T_{\epsilon_1} > T_\epsilon$ such that
	$$u(t, x) \geq \frac{1}{2}W_\epsilon(x) > 0, ~\forall ~t \geq T_{\epsilon_1}, ~x\in[g_\infty + \epsilon, h_\infty - \epsilon].$$
	It follows that, for $t\geq T_{\epsilon_1}$, 
	\begin{align*}
	h'(t) &= \mu\int_{g(t)}^{h(t)}\int_{h(t)}^{\infty}J_1(x - y)u(t, x)\rd y\rd x\\
	&\geq \mu \int_{h_\infty-2\epsilon}^{h_\infty-\epsilon}\int_{h_\infty}^{h_\infty+\epsilon} J_1(x-y)\frac 12 W_\epsilon(x) dy dx\\
	&\geq \mu \int_{h_\infty-2\epsilon}^{h_\infty-\epsilon}\int_{h_\infty}^{h_\infty+\epsilon} \min_{z\in [-3\epsilon, -\epsilon]}J_1(z)\frac 12 \min_{z\in [0, h_\infty]}W_\epsilon(z) dy dx
	\\
	&=\mu \epsilon^2 \min_{z\in [-3\epsilon, -\epsilon]}J_1(z)\frac 12 \min_{z\in [0, h_\infty]}W_\epsilon (z)=:\sigma_\epsilon>0
	\end{align*}
	provided that $\epsilon>0$ is sufficiently small (recall that $J_1(0)>0$). But this implies $h_\infty=\infty$, a contradiction.
	 Thus $\lambda_p(\mathcal{L}_{(g_\infty, h_\infty)})\leq k-1$ holds true. This proves the first inequality in \eqref{3.3}.
	 
The first inequality in \eqref{3.3} is a consequence of the second. Indeed,  by  Proposition B, we have $-d_1<\lambda_p(\mathcal{L}_{(g_\infty, h_\infty)} )\leq k-1$, i.e., $d_1>1-k$.  The proof is finished. 
\end{proof}

Our analysis below will make use of Barbalat's Lemma \cite{Ba}, which is recalled below and can be proved by elementary calculus directly.

\begin{lemma}[Barbalat's Lemma]\label{lemmaba} Suppose that $\psi:[0,\yy)\rightarrow \mathbb{R}$ is uniformly continuous
	and that $\dd\lim\limits_{t\to\yy}\int_{0}^t \psi(s)\rd s\in\R$ exists. Then $\dd\lim_{t\to\yy} \psi(t)=0$.
\end{lemma}

\begin{lemma}\label{lemma3.5}  If  $h_\infty - g_\infty < \infty$, then
	\begin{align}\label{3.5}
		\lim\limits_{t \rightarrow \infty}\int_{\R}u(t, y)\rd y = 0.
	\end{align}
\end{lemma}
\begin{proof}
By the argument in the proof of Lemma 3.1 in \cite{DY-2022-DCDS},  we obtain that $\lim\limits_{t \rightarrow \infty}g^\prime(t) = \lim\limits_{t \rightarrow \infty}h^\prime(t) = 0$.
Due to $J_i(0)>0$,  there is small  $\epsilon > 0$ such that $\inf\limits_{x\in [-\epsilon,\epsilon]} J_1(x)>0$. Then for large $t>0$,
\begin{align*}
	h^\prime(t)& = \mu\int_{g(t)}^{h(t)}\int_{h(t)}^{\infty}J_1(x - y)u(t, x)\rd y \rd x\\
	&\geq \mu\int_{h(t)-\epsilon /2}^{h(t)}\int_{h(t)}^{h(t)+\epsilon /2}J_1(x - y)u(t, x)\rd y \rd x\\
	& = \mu\int_{h(t) - \epsilon /2}^{h(t)}\int^{x - h(t)}_{x - (h(t) + \epsilon /2)}J_1(z)u(t, x)\rd z\rd x\\
	&\geq \mu\frac{\epsilon}{2}\inf_{x\in [-\epsilon,\epsilon]} J_1(x)\int_{h(t)-\epsilon /2}^{h(t)}u(t, x)\rd x\geq 0,
\end{align*} 
which implies 
\begin{align}\label{3.6}
\lim_{t\to\yy}\int_{h(t)-\epsilon /2}^{h(t)}u(t, x)\rd x=0\ \ {\rm and\ hence}\ \  \lim_{t\to\yy}\int_{h_\yy-\epsilon /4}^{\infty}u(t, x)\rd x  = 0.
\end{align}
Similarly, using $\lim\limits_{t \rightarrow \infty}g^\prime(t)=0$ we deduce
\begin{align}\label{3.7}
\lim_{t\to\yy}\int_{g (t)}^{g (t)+\epsilon /2}u(t, x)\rd x=0\ \ {\rm and\ hence}\ \  \lim_{t\to\yy}\int_{-\infty}^{g_\yy+\epsilon /4}u(t, x)\rd x  = 0.
\end{align}

In the following, we show that 
\begin{align}\label{3.8}
\mbox{ $\lim_{t\to \yy} U(t,x)=0$ for every $x\in \R$,}
\end{align}
where
\[
	U(t,x):=\int_{x-L}^{x+L}u(t,y)\rd y,\ L= \epsilon/8.
\]

Due to \eqref{3.6}, \eqref{3.7} and the fact that $u(t,x)\equiv 0$ for $t\geq 0$ and $x\in \R\backslash [g (t),h(t)]$, clearly \eqref{3.8} holds for $x\in \R\backslash (g_\yy+3\epsilon/16, h_\yy-3\epsilon/16)$. It remains to show \eqref{3.8} for  $x\in(g_\yy+3\epsilon/16, h_\yy-3\epsilon/16)$.

In view of the equation satisfied by $u$, we deduce for large $t$ and $x\in [g_\yy+3\epsilon/16, h_\yy-3\epsilon/16]$,
\begin{align}\label{U}
	U_t= d_1\int_{x-L}^{x+L}\int_{g(t)}^{h(t)}J_1(y-z)u(t, z)\rd z\rd y+\int_{x-L}^{x+L}[- d_1u + u(1 - u - kv)]\rd y. 
\end{align} 
From the boundedness of $u$ and $v$, we easily see from the above equation that  $|U_t(t,x)|$ is uniformly bounded  for all large $t$ and $x\in [g_\yy+3\epsilon/16, h_\yy-3\epsilon/16]$.  Hence for any fixed $x\in [g_\yy+3\epsilon/16, h_\yy-3\epsilon/16]$, $t\to U(t,x)$ is uniformly continuous in $t$ for all large $t$, and hence for all $t\in [0,\infty)$.

Now take $x_0=h_\yy-3\epsilon/16$  and recall that \eqref{3.8} holds with $x=x_0$; so we have
\[
\lim_{t\to\infty}\int_0^t U_t(s,x_0)ds=\lim_{t\to\infty} U(t, x_0)-U(0,x_0)=-U(0, x_0).
\]
We may now apply  Lemma \ref{lemmaba} with $\psi(t)=U_t(t,x_0)$ to conclude that 
$\lim_{t\to\yy}U_t(t,x_0)=0$. Moreover, using \eqref{3.8} with $x=x_0$ we further obtain
\begin{align*}
	\lim_{t\to\yy}\int_{x_0-L}^{x_0+L}[- d_1u + u(1 - u - kv)]\rd y=0.
\end{align*}
Hence we can use \eqref{U} to deduce
\begin{align*}
	\lim_{t\to\yy}d_1\int_{x_0-L}^{x_0+L}\int_{g(t)}^{h(t)}J_1(y-z)u(t, z)\rd z\rd y=0.
\end{align*}
It follows, due to $\inf_{x\in [-\epsilon,\epsilon]} J_1(x)>0$ and $L= \epsilon/8$, that
\begin{align*}
	0=&\ \lim_{t\to\yy}\int_{x_0-L}^{x_0+L}\int_{g(t)}^{ h(t)}J_1(y-z)u(t, z)\rd z\rd y\geq 	\limsup\limits_{t\to\yy}\int_{x_0-L}^{x_0+L}\int_{x_0-2L}^{x_0+2L}J_1(y-z)u(t, z)\rd z\rd y\\
	\geq&\ 2L\inf_{x\in [-\epsilon,\epsilon]} J_1(x)\limsup\limits_{t\to\yy}\int_{x_0-2L}^{x_0+2L}u(t, z)\rd z\geq 0,
\end{align*}
which implies
\begin{align*}
	\lim_{t\to\yy}\int_{x_0-2L}^{x_0+2L}u(t, z)\rd z=0.
\end{align*}
In particular, $\lim_{t\to\yy} U(t,x)=0$ for all $x\in [x_0-L,x_0]$. 

Now we may repeat  the above argument with $x_0$ replaced by $x_1:=x_0-L$,  and  similarly  show  $\lim_{t\to\infty}U(t,x)=0$ for $x\in [x_1-L, x_1]=[x_0-2L, x_0-L]$. 

Analogously we may take $y_0=g_\yy+3\epsilon/16$ and show that \eqref{3.8} holds for $x\in [y_0, y_0+L]$ and then continue to show that \eqref{3.8} holds for
$x\in [y_0+L, y_0+2L]$, etc.

Clearly after finitely many steps we reach the conclusion that \eqref{3.8} holds for all $x\in [y_0, x_0]$, as desired. We have thus proved that \eqref{3.8} holds for every $x\in\R$, namely
\[
\lim_{t\to\infty} \int_{x-L}^{x+L} u(t,y)dy=0 \mbox{ for every } x\in\R.
\]
Since $u(t,y)=0$ for all $t>0$ and $y\not\in[g_\infty, h_\infty]$, this implies \eqref{3.5}.
The proof of the lemma is now complete.
\end{proof}
\begin{remark}\label{remark3.6}
If $g_\yy=-\infty$ and $h_\yy<+\infty$, then we could take $x_1=x_0-L$, ... $x_{n+1}=x_n-L$ and repeat the argument in the proof of  Lemma {\rm \ref{lemma3.5}}
finitely many times to obtain 
\[
\lim_{t\to\infty}\int_{-M}^\infty u(t,x)dx=0 \mbox{ for every $M>0$.}
\]
When $g_\yy>-\infty$ and $h_\yy=+\infty$, we can show
\[
\lim_{t\to\infty}\int^{M}_{-\infty} u(t,x)dx=0 \mbox{ for every $M>0$.}
\]
\end{remark}

To prove the next lemma, we will need a trick introduced  in the proof of Theorem 3.3 in \cite{DY-2022-DCDS}, which is formulated
 in a more general form below.

\begin{lemma}\label{general}
Suppose that $s_1(t)$ and $s_2(t)$ are continuous bounded functions over $[0,\infty)$ satisfying $s_1(t)<s_2(t)$ for all $t\geq 0$.
Let $U(t,x)$ be a continuous bounded function over  $\Omega:=\{(t,x): t\geq 0,\ x\in [s_1(t), s_2(t)]\}$, with $U_t(t,x)$ also continuous in $\Omega$. Then there exist sequences
$(\underline t_n, \underline x_n)$ and $(\bar t_n,\bar x_n)$ with $\underline x_n\in [s_1(\underline t_n), s_2(\underline t_n]$, $\bar x_n\in [s_1(\bar t_n), s_2(\bar t_n)]$
and $\lim_{n\to\infty} \underline t_n=\lim_{n\to\infty} \bar t_n=\infty$ such that
\[\begin{cases}
\lim_{n\to\infty} U(\underline t_n, \underline x_n)=\underline U,\ \lim_{n\to\infty}U_t(\underline t_n, \underline x_n)=0,\\
\lim_{n\to\infty} U(\bar t_n, \bar x_n)=\overline U,\ \lim_{n\to\infty}U_t(\bar t_n, \bar x_n)=0,
\end{cases}
\]
where
\[
\underline U:= \liminf_{t\to\infty}\min_{x\in [s_1(t), s_2(t)]}U(t,x),\ \overline U:= \limsup_{t\to\infty}\max_{x\in [s_1(t), s_2(t)]}U(t,x).
\]
\end{lemma}
\begin{proof}
We only prove the existence of $(\bar t_n, \bar x_n)$ since the existence of $(\underline t_n, \underline x_n)$ then follows by considering the function $V(t,x)=-U(t,x)$.

Denote
 \[M(t):=\max_{x\in [s_1(t), s_2(t)]}U(t,x) \mbox{ and } X(t):=\{x\in [s_1(t), s_2(t)]: U(t,x)=M(t)\}.\]
Then $X(t)$ is a compact set for each $t>0$. Therefore, there exist $\underline\xi(t),\;\overline\xi(t)\in X(t)$ such that
  \[U_{t}(t, \underline\xi(t))=\min_{x\in X(t)}U_{t}(t,x),\quad U_{t}(t, \overline\xi(t))=\max_{x\in X(t)}U_{t}(t,x).\]
We claim that $M(t)$ satisfies, for each $t>0$,
\bes\label{3.8a}
\left\{\begin{array}{ll}
M'(t+0):=\dd\lim_{s>t, s\to t}\frac{M(s)-M(t)}{s-t} =U_{t}(t,\overline\xi(t)),\\[3mm]
 M'(t-0):=\dd\lim_{s<t, s\to t}\frac{M(s)-M(t)}{s-t} =U_{t}(t,\underline\xi(t)).
 \end{array}\right.
\ees
Indeed, for any fixed $t>0$ and $s>t$, we have
 \[ U(s,\overline\xi(t))-U(t,\overline\xi(t))\leq M(s)-M(t)\leq  U(s,\overline \xi(s))-U(t,\overline\xi (s)).\]
It follows that
 \bes\label{geq}
\liminf_{s>t, s\to t}\frac{M(s)-M(t)}{s-t} \geq U_{t}(t,\overline\xi(t)),
 \ees
and
 \[\limsup_{s>t, s\to t}\frac{M(s)-M(t)}{s-t}\leq \limsup_{s>t, s\to t} \frac{U(s,\overline \xi(s))-U(t,\overline\xi(s))}{s-t}.\]
Let $s_n\searrow t$ satisfy
  \[\lim_{n\to\infty}\frac{U(s_n,\overline \xi(s_n))-U(t,\overline\xi(s_n))}{s_n-t}= \limsup_{s>t, s\to t} \frac{U(s,\overline \xi(s))-U(t,\overline\xi(s))}{s-t}.\]
By passing to a subsequence if necessary, we may assume that $\overline \xi(s_n)\to \xi$ as $n\to\infty$. Then $U(t, \xi)=\dd\lim_{n\to\infty} M(s_n)=M(t)$ and hence
$\xi\in X(t)$.
Due to the continuity of $U_{t}(t,x)$, it follows immediately that
 \[\lim_{n\to\infty}\frac{U(s_n,\overline \xi(s_n))-U(t,\overline\xi(s_n))}{s_n-t}=U_{t}(t,\xi)\leq U_{t}(t,\overline \xi(t)).\]
We thus obtain
 \[\limsup_{s>t, s\to t}\frac{M(s)-M(t)}{s-t}\leq U_{t}(t,\overline \xi_i(t)).\]
Combining this with \eqref{geq} we obtain
 \[M'(t+0)=U_{t}(t,\overline \xi(t)).\]
Analogously we can show
 \[M'(t-0)=U_{t}(t,\underline \xi(t)).\]
Let us note from \eqref{3.8a} that $M'(t-0)\leq M'(t+0)$ for all $t>0$. Therefore if $M(t)$ has a local maximum at $t=t_0$, then
$M'(t_0)$ exists and $M'(t_0)=0$. 

Regarding the function $M(t)$ we have three possibilities:

  (a)  it has a sequence of local maxima $\{t_n\}$ such that 
  \[
  \mbox{$\lim_{n\to\infty} t_n=\infty$ and $\lim_{n\to\infty} M(t_n)=\limsup_{t\to\infty} M(t)$,}
  \]
  
  (b) it is monotone nondecreasing for all large $t$ and so $\lim_{t\to\infty} M(t)$ exists,
  
  (c)  it is monotone nonincreasing for all large $t$ and so $\lim_{t\to\infty} M(t)$ exists.
  \medskip
  
 In case (a) we take $\bar t_n=t_n,\ \bar x_n=\overline \xi(t_n)$ and  so
 \[
 U_t(\bar t_n, \bar x_n)=M'(t_n)=0,\ U(\bar t_n, \bar x_n)=U(t_n, \overline \xi(t_n))=M(t_n)\to \limsup_{t\to\infty} M(t) \mbox{ as } n\to\infty.
 \]
 
 In case (b) 
necessarily $M'(t_n-0)\to 0$ along some sequence $t_n\to\infty$ for otherwise $M'(t+0)\geq M'(t-0)\geq \delta >0$ for some $\delta>0$ and all large $t$, which leads to the contradiction 
$M(t)\to\infty$ as $t\to\infty$. We now take $\bar t_n=t_n$ and $\bar x_n=\underline\xi(t_n)$, and obtain
\[\begin{cases}
U_t(\bar t_n,\bar x_n)=U_t(t_n, \underline \xi(t_n))=M'(t_n-0)\to 0,\\ U(\bar t_n, \bar x_n)=U(t_n, \underline \xi(t_n))=M(t_n)\to\lim_{t\to\infty} M(t) \mbox{ as } n\to\infty.
\end{cases}\]

In case (c), 
necessarily $M'(s_n+0)\to 0$ along some sequence $s_n\to\infty$ for otherwise $M'(t-0)\leq M'(t+0)\leq -\delta <0$ for some $\delta>0$ and all large $t$, which leads to the contradiction
$M(t)\to -\infty$ as $t\to\infty$. We now take $\bar t_n=s_n$ and $\bar x_n=\overline\xi(s_n)$, and obtain
\[\begin{cases}
U_t(\bar t_n,\bar x_n)=U_t(s_n, \overline \xi(s_n))=M'(s_n+0)\to 0,\\ U(\bar t_n, \bar x_n)=U(s_n, \overline \xi(s_n))=M(s_n)\to\lim_{t\to\infty} M(t) \mbox{ as } n\to\infty.
\end{cases}\]
The proof is complete.
\end{proof}

\begin{lemma}\label{l2.7} If  $h_\infty - g_\infty < \infty$, then for any given $L>0$, we have
	\begin{align}\label{3.9}
		\liminf\limits_{t\to \infty} \min_{x\in [-L,L]}v(t,x)>0,
	\end{align}
	and for any $\epsilon>0$, there is $L_*=L_*(\epsilon, L)\gg 1$ such that  $L_1\geq L_*$ leads to
	\begin{align}\label{3.10}
	\liminf\limits_{t\to\yy} \min_{|x|\in [L_1,L_1+L]}v(t,x)\geq 1-\epsilon,
	\end{align}
which  implies
\begin{align*}
		\liminf\limits_{t\to\yy} v(t,x)\geq 1-\epsilon \mbox{ for every}\ x\geq L_1.
\end{align*}
\end{lemma}
\begin{proof} {\bf Step 1.} We first show \eqref{3.10}, and  only consider the case of $x>0$, as the case $x<0$ can be treated similarly. 
	
	Recalling that $u(t,x)=0$ for all $x\not\in [g_\yy,h_\yy]$, we see that $v$ satisfies 
\begin{align*}
v_t \geq d_2\int_{h_\yy}^{\yy}J_2(x - y)v(t, y)\rd y - d_2v + \gamma v(1 - v ),  \
t > 0,~x\in [h_\yy,\yy).
\end{align*}

To prove \eqref{3.10}, 
we will utilize the  conclusions of \cite[Propositions 3.5 and 3.6]{Cao-2019-JFA} about the solution of a nonlocal diffusion problem over a fixed spatial interval.  To be precise, for  constants $a<b$, let $w(t,x)$ be the solution of the following problem 
\begin{align*}
	w_t = d_2\int_{a}^{b}J_2(x - y)w(t, y)\rd y - d_2w+ \gamma w(1 - w )
\ \ \	t > 0,~x\in (a,b),
\end{align*}
with continuous initial function $w(0,x)\geq,\not\equiv 0$ in $[a,b]$. Then by \cite[Proposition 3.5]{Cao-2019-JFA}, for sufficient large $b-a$, the solution $w(t,x)$ converges to $w_{ab}(x)$ uniformly for $x\in [a,b]$ as $t\to\yy$  with $\lambda_p(\mathcal{L}_{(a, b)}+ \gamma) > 0$, where $w_{ab}(x)$ satisfies
\begin{align*}
	d_2\int_{a}^{b}J_2(x - y)w_{ab}(y)\rd y - d_2w_{ab}+ \gamma w_{ab}(1 - w_{ab} )=0,  
	\ \ \  x\in (a,b).
\end{align*}
\cite[Proposition 3.6]{Cao-2019-JFA} then asserts
\begin{align*}
	\lim_{a\to-\yy,\ b\to\yy} w_{ab}(x)=1\ \ {\rm locally\ uniformly\ in}\ \R,
\end{align*}
Hence, there is  large $L_*>L$ such that for $-a,b\geq  L_*$,
\begin{align*}
	w_{ab}(x)>1-\epsilon/2 \mbox{ in }  [-L,L].
	\end{align*}	
Recalling from \cite[Proposition 3.6]{Cao-2019-JFA} that  $w_{ab}$  has the  shifting invariance property:  for $\td a=a+\delta$ and $\td b=b+\delta$, there holds $ w_{\td a \td b}(x)=w_{ab}(x-\delta)$ for $x\in [\td a,\td b]$. Taking $-a=b=\tilde L_1\geq L_*$, $\td a=h_\yy$ and $\td b= 2\tilde L_1+h_\yy$, we obtain
\begin{align*}
	w_{\td a\td b}(x)\geq 1-\epsilon/2 \mbox{ for } \ x\in [L_1,L_1+L]
\end{align*}
with  $L_1:=\tilde L_1+h_\yy$, due to $w_{\td a\td b}(x)=w_{ab}(x-h_\yy-\td L_1)$ for $x\in [\td a,\td b]$.
Then, the convergence  $\lim_{t\to\infty}w(t,x)=w_{\tilde a\tilde b}(x)$ as a solution defined on $(t,x)\in [0,\yy)\times [\td a, \td b]$ gives 
	\begin{align}\label{3.11}
	\liminf\limits_{t\to\yy} \min_{|x|\in [L_1,L_1+L]}w(t,x)\geq 1-\epsilon.
\end{align}
 Letting $w(0,x):=v(1,x)$ for $x\in [\td a,\td b]=[h_\yy,h_\yy+2\td L_1]$, by the comparison principle, we obtain
\begin{align}\label{v-w}
	v(t,x)\geq w(t,x) \ \mbox{ for }  \ t\geq 0,\ x\in  [h_\yy,h_\yy+2\td L_1],
\end{align}
and so \eqref{3.10} follows from \eqref{3.11}. 
\medskip

{\bf Step 2.} We now verify \eqref{3.9}.  

From \eqref{v-w} we see that for any $[A,B]\subset \R\backslash (g_\yy,h_\yy)$,
\begin{align*}
	\liminf\limits_{t\to \infty} \min_{x\in [A,B]}v(t,x)>0,
\end{align*}
Hence, to show the validity of \eqref{3.9}, we only need to check
\begin{align}\label{3.12}
	\liminf\limits_{t\to \infty} \min_{x\in [g_\yy,h_\yy]}v(t,x)>0.
\end{align} 

Denote 
\begin{align*}
\delta_0:=\liminf\limits_{t\to \yy} \min_{x\in [h_\yy,h_\yy+1]}v(t,x)>0
\end{align*} 
and choose $\epsilon_0>0$ small so that $J_2(x)>0$ for $x\in [-2\epsilon_0, 2\epsilon_0]$.
 We show that 
\begin{align*}
	\delta_1:=\liminf\limits_{t\to \yy} \min_{x\in [h_\yy-\epsilon_0,h_\yy]}v(t,x)>0.
\end{align*}
Suppose, on the contrary, $\delta_1=0$, we are going to derive a contradiction. Using $\delta_1=0$ and Lemma \ref{general}, we can find two sequences $t_n \rightarrow \infty$ and $x_n\in [h_\yy-\epsilon_0,h_\yy]$ satisfying $x_n\to x_0\in [h_\yy-\epsilon_0,h_\yy]$ such that 
\begin{align*}
	\lim\limits_{n\rightarrow \infty}v(t_n,x_n) = 0, ~\lim\limits_{n\rightarrow \infty}v_t(t_n, x_n) = 0.
\end{align*}
Then, letting $(t,x)=(t_n,x_n)$ and $n\to\yy$ in the equation satisfied by $v$, we obtain
\begin{align*}
	0=&\lim_{n\to \yy} \int_{\R}J_2(x_n - y)v(t_n, y)\rd y\geq \liminf\limits_{n\to \yy} \int_{h_\yy}^{h_\yy+1}J_2(x_n - y)v(t_n, y)\rd y\\
	\geq &\ \delta_0\int_{h_\yy}^{h_\yy+1}J_2(x_0 - y)\rd y>0,
\end{align*}
where we have used Fatou's Lemma in the third inequality.  This is a contradiction. Hence $\delta_1>0$. 

We may now similarly show that 
\begin{align*}
	\delta_2=\liminf\limits_{t\to \yy} \min_{x\in [h_\yy-2\epsilon_0,h_\yy-\epsilon_0]}v(t,x)>0.
\end{align*}
Repeating the argument several times, one can obtain the desired inequality  \eqref{3.12}.   The proof is completed.  
\end{proof}

\begin{lemma}\label{lemma3.7} If  $h_\infty - g_\infty < \infty$, then  for every $L>0$,  we have
	\begin{align}\label{3.13}
		\lim_{t\to\yy}\int_{-L}^{L}|v(t,y)-1| \rd y= 0
	\end{align}
and 
\begin{align}\label{3.14}
	&\lim_{t\to\yy}v(t,x)\to 1\ \ \ {\rm  uniformly\ for} \ x\in [-L, L]\backslash  (g_\yy,h_\yy).
\end{align}
	
\end{lemma}
\begin{proof} {\bf Step 1.} We  prove that
		\begin{align}\label{3.15}
			\lim_{t\to\yy}V(t,x)= 1 \mbox{ for every $x\in \R$,}
		\end{align}
 where
	\begin{align*}
		  V(t,x):=\frac{1}{2L}\int_{x-L}^{x+L}v(t,y) \rd y.
	\end{align*} 

Since
\begin{align*}
	&\frac{1}{2L}\int_{x-L}^{x+L}\int_\mathbb{R}J_2(y - z)v(t, z)\rd z\rd y=\frac{1}{2L}\int_{x-L}^{x+L}\int_\mathbb{R}J_2(z)v(t, y-z)\rd z\rd y\\
	=&\frac{1}{2L}\int_\mathbb{R}J_2(z)\lf(\int_{x-L}^{x+L}v(t, y-z)\rd y\rr)\rd z=\int_\mathbb{R}J_2(z)V(t,x-z)\rd z\nonumber\\
	=&\int_\mathbb{R}J_2(x-y)V(t,y)\rd y,\nonumber
\end{align*}
by the equation for $v$, the function  $V(t,x)$ satisfies
\begin{equation}\begin{aligned}\label{3.16}
	V_t=& \ d_2\frac{1}{2L}\int_{x-L}^{x+L}\int_\mathbb{R}J_2(y - z)v(t, z)\rd z\rd y- d_2V + \frac{1}{2L}\int_{x-L}^{x+L}\gamma v(1 - v - hu)\rd y\\
	=&\  d_2\int_\mathbb{R}J_2(x-y)V(t,y)\rd y- d_2V + \frac{1}{2L}\int_{x-L}^{x+L}\gamma v(1 - v - hu)\rd y \mbox{ for } t>0,\ x\in \R.
\end{aligned}
\end{equation}

To show \eqref{3.15}, it suffices to verify for any given $0<  \delta<1$, 
\begin{align*}
\liminf\limits_{t\to\yy} V(t,x)=\liminf\limits_{t\to\yy}\frac{1}{2L}\int_{x-L}^{x+L}v(t,y) \rd y\geq \delta \mbox{ for every}  \ x\in \R.
\end{align*}

We first show that this is true when $|x|$ is large. 
  In fact, by \eqref{3.10}, for any $\epsilon\in (0,1-\delta)$, any $a>0$ and all large $L_1>0$, say $L_1\geq L^*(a)>0$,
	\begin{align*}
	\liminf\limits_{t\to\yy} \min_{|x|\in [L_1,L_1+a]}v(t,x)\geq 1-\epsilon>\delta.
\end{align*}
Hence, for  $|x|\geq L^*(2L)+L$, by Fatou's Lemma, 
\begin{align}\label{3.18}
	\liminf\limits_{t\to\yy} V(t,x)=\liminf\limits_{t\to\yy}\frac{1}{2L}\int_{x-L}^{x+L}v(t,y) \rd y\geq \frac{1}{2L}\int_{x-L}^{x+L}\liminf\limits_{t\to\yy}v(t,y) \rd y>\delta.
\end{align}

For $|x|\leq L^*(2L)+L$, the desired conclusion is a consequence of the following stronger result:
\begin{align}\label{V-1}
	\liminf\limits_{t\to \yy}\min_{|x|\leq L_1}V(t,x)\geq \delta \mbox{ for every $L_1>0$}.
\end{align}
Otherwise, there exists $L_1>0$ such that
\[
\delta_1:=\liminf\limits_{t\to \yy}\min_{|x|\leq L_1}V(t,x)<\delta.
\]
Without loss of generality, we may assume that $L_1>L^*(2L)+L$.
 By Lemma \ref{general}  we can choose  two sequences $t_n \rightarrow \infty$ and $x_n\in [-L_1,L_1]$ satisfying $x_n\to  x_0\in [-L_1,L_1]$ such that 
\begin{align*}
	\lim\limits_{n\rightarrow \infty}V(t_n,x_n) = \delta_1, ~\lim\limits_{n\rightarrow \infty}V_t(t_n, x_n) = 0.
\end{align*}
We may also require, by passing to a subsequence if necessary, that the following limits exist
\begin{align*}
	\lim_{n\to\yy}\int_\mathbb{R}J_2(x_n - y)V(t_n, y)\rd y,\ \ \lim_{n\to\yy}\frac{1}{2L}\int_{x_n-L}^{x_n+L}\gamma v(t_n,y)[1-v(t_n,y)]\rd y.
\end{align*}
Then from \eqref{3.16} we deduce, upon using \eqref{3.5},
\begin{align*}
	0=&\lim_{n\to\yy}d_2\int_\mathbb{R}J_2(x_n - y)V(t_n, y)\rd y-d_2\delta_1+\lim_{n\to\yy}\frac{\gamma}{2L}\int_{x_n-L}^{x_n+L} v(t_n,y)[1-v(t_n,y)]\rd y\\
	\geq&\ d_2\int_\mathbb{R}\liminf\limits_{n\to\yy} \lf[J_2(x_n - y)V(t_n, y)\rr]\rd y-d_2\delta_1\\
	&+\frac{\gamma}{2L}\int_{-L}^{L} \liminf\limits_{n\to\yy}\big(v(t_n,y+x_n)[1-v(t_n,y+x_n)]\big)\rd y.
\end{align*}
where Fatou's Lemma is used in the last inequality. In view of the definition of $\delta_1$ and \eqref{3.18}, we have 
\begin{align*}
	d_2\int_\mathbb{R}\liminf\limits_{n\to \yy}J_2(x_n - y)V(t_n, y)\rd y-d_2\delta_1
	\geq	d_2\int_\mathbb{R}J_2( x_0 - y)\delta_1\rd y-d_2\delta_1= 0.
\end{align*}
From $v\geq 0$ and $\limsup\limits_{t\to\yy} v(t,y)\leq 1$, we obtain
\begin{align*}
\int_{-L}^{L}\liminf\limits_{n\to\yy}\big(v(t_n,y+x_n)[1-v(t_n,y+x_n)]\big)\rd y\geq 0.
\end{align*}
Thus we have
\[
0=\lim_{n\to\yy}d_2\int_\mathbb{R}J_2(x_n - y)V(t_n, y)\rd y-d_2\delta_1+\lim_{n\to\yy}\frac{\gamma}{2L}\int_{x_n-L}^{x_n+L} v(t_n,y)[1-v(t_n,y)]\rd y\geq 0,
\]
 which implies
\begin{align}\label{3.18a}
	&\lim_{n\to\yy}\int_\mathbb{R}J_2(x_n - y)V(t_n, y)\rd y-d_2\delta_1=0\nonumber\\
	&\lim_{n\to\yy}\int_{x_n-L}^{x_n+L} v(t_n,y)[1-v(t_n,y)]\rd y=0.
\end{align}

Moreover, comparing $v$ with the solution $w$  of the ODE 
\[
w'=\gamma w(1-w),\ w(0)=k_1:=\max\limits_{x\in \R}{ v}_0(x)+1,
\]
we obtain
\begin{align}\label{3.19}
	v(t,x)\leq w(t)\leq 1 + (k_1 - 1)e^{-\gamma t} \mbox{ for }  t\geq 0,\ x\in \R.
\end{align}

{\bf Claim}. For small $\epsilon_1>0$,
\begin{align}\label{3.20}
		\lim_{n\to\yy}\int_{x_0-L+\epsilon_1}^{x_0+L-\epsilon_1} [1-v(t_n,y)]\rd y=0.
\end{align}

In view of \eqref{3.18a}, $x_n\to x_0$, and $\liminf\limits_{n\to\yy}v(t_n,y+x_n)[1-v(t_n,y+x_n)]\geq 0$, we have
\begin{align}\label{3.21}
\lim_{n\to\yy}\int_{x_0-L+\epsilon_1}^{x_0+L-\epsilon_1}v(t_n,y)[1-v(t_n,y)]\rd y=0.
\end{align}
Denote 
\begin{align*}
	&\Omega_1(t_n):=\{x\in [x_0-L+\epsilon_1,x_0+L-\epsilon_1]: v(t_n,x)\leq 1\},\\ 
	&\Omega_2(t_n):=\{x\in [x_0-L+\epsilon_1,x_0+L-\epsilon_1]: v(t_n,x)> 1\}.
\end{align*}
Then clearly
\begin{align*}
	&\int_{x_0-L+\epsilon_1}^{x_0+L-\epsilon_1}v(t_n,y)[1-v(t_n,y)]\rd y\\
	=&\int_{\Omega_1(t_n)}v(t_n,y)[1-v(t_n,y)]\rd y+\int_{\Omega_2(t_n)}v(t_n,y)[1-v(t_n,y)]\rd y\\
	\geq &\ k_0\int_{\Omega_1(t_n)}[1-v(t_n,y)]\rd y+\int_{\Omega_2(t_n)}v(t_n,y)[1-v(t_n,y)]\rd y,
\end{align*}
where $k_0:=\min_{\{t\geq 1,\ y\in[x_0-L+\epsilon_1, x_0+L-\epsilon_1]\}}v(t,y)>0$ by \eqref{3.9}. 

By \eqref{3.19}, 
\begin{align*}
	\lim\limits_{n\to\yy}\max_{y\in \Omega_2(t_n)} |1-v(t_n,y)|= 0,
\end{align*}
which implies
\begin{align*}
	\lim_{n\to\yy}\int_{\Omega_2(t_n)}[1-v(t_n,y)]\rd y=\lim_{n\to\yy}\int_{\Omega_2(t_n)}v(t_n,y)[1-v(t_n,y)]\rd y=0.
\end{align*}
Therefore, as $n\to\infty$,  by \eqref{3.21},
\begin{align*}
0&\leq k_0\int_{\Omega_1(t_n)}[1-v(t_n,y)]\rd y\\
&\leq \int_{x_0-L+\epsilon_1}^{x_0+L-\epsilon_1}v(t_n,y)[1-v(t_n,y)]\rd y-\int_{\Omega_2(t_n)}v(t_n,y)[1-v(t_n,y)]\rd y\to 0.
\end{align*}
It follows that
\[
\lim_{n\to\infty}\int_{\Omega_1(t_n)}[1-v(t_n,y)]\rd y=0.
\]
Hence
\[
\lim_{n\to\yy}\int_{x_0-L+\epsilon_1}^{x_0+L-\epsilon_1} [1-v(t_n,y)]\rd y=\lim_{n\to\infty}\left[\int_{\Omega_1(t_n)}[1-v(t_n,y)]\rd y+\int_{\Omega_2(t_n)}[1-v(t_n,y)]\rd y\right]=0.
\]
This proves \eqref{3.20}.

Now we are ready to use the above information to get a contradiction. Recalling the definition of $V$ and $x_n\to x_0$ as $n\to \yy$, one sees that for large $n$,
\begin{align*}
	V(t_n,x_n)=\frac{1}{2L}\int_{x_n-L}^{x_n+L}v(t_n,y) \rd y
	\geq \frac{1}{2L}\int_{x_0-L+\epsilon_1}^{x_0+L-\epsilon_1} v(t_n,y)\rd y.
\end{align*}
Then by  \eqref{3.20} we obtain, for sufficiently small $\epsilon_1>0$, 
\begin{align*}
	\lim_{n\to \yy}	V(t_n,x_n)\geq \lim_{n\to \yy} \frac{1}{2L}\int_{x_0-L+\epsilon_1}^{x_0+L-\epsilon_1} v(t_n,y)\rd y=  \frac{L-\epsilon_1}{L}>\delta_1,
\end{align*}
which contradicts to the fact $\lim\limits_{n\rightarrow \infty}V(t_n,x_n) = \delta_1$.
Therefore \eqref{3.15} holds. 

{\bf Step 2.}  We show \eqref{3.13}.

 We first observe that \eqref{V-1} and \eqref{3.19} imply 
\begin{align}\label{3.22}
	\lim_{t\to \yy}V(t,x)=1\ \  \ {\rm locally\ uniformly\ for}\ x\in \R.
\end{align}

Next by a simple computation we have
\begin{align*}
	\frac{1}{2L}\int_{-L}^{L}|v(t,y)-1| \rd y=1-V(t, 0)+\frac{1}{L}\int_{\Omega_2}[v(t,y)-1] \rd y
\end{align*}
where $\Omega_2=\Omega_2(t):=\{y\in [-L,L]: v(t,y)> 1\}$. Now \eqref{3.13} follows directly from \eqref{3.22} and
 \eqref{3.19}. 

{\bf Step 3.} We prove \eqref{3.14}.  

Since $u(t,x)\equiv 0$ for $x\not\in (g_\yy,h_\yy)$,  the function $v$ satisfies, for any interval $[a,b]\subset  \R\backslash  (g_\yy,h_\yy)$,
\begin{align}\label{v-[a,b]}
		v_t = d_2\int_{\R}J_2(x - y)v(t, y)\rd y - d_2v + \gamma v(1 - v ),  \ \ 
		t > 0,~x\in [a, b].
	\end{align}
 Denote 
\begin{align*}
	\delta_2=\liminf\limits_{t\to \yy} \min_{x\in [a,b]}v(t,x).
\end{align*}
To show \eqref{3.14}, we only need to prove $\delta_2=1$ due to the arbitrariness of $[a, b]$. 

By Lemma \ref{general} we can choose  two sequences $t_n \rightarrow \infty$ and $x_n\in [a,b]$ satisfying $x_n\to x_0\in [a,b]$  as $n \to \yy$ such that 
\begin{align*}
	\lim\limits_{n\rightarrow \infty}v(t_n,x_n) = \delta_2, ~\lim\limits_{n\rightarrow \infty}v_t(t_n, x_n) = 0.
\end{align*}
We may also require that the following limit exists
\begin{align*}
		B:=\lim_{n\to\yy}\int_\mathbb{R}J_2(x_n - y)v(t_n, y)\rd y.
\end{align*}
We claim that $B=1$. 
Otherwise, $B\neq 1$. Since $\limsup\limits_{t\to\yy} v\leq 1$, we must have $B<1$. Due to $\int_{\R}J_2(x)\rd x=1$, there is large $\hat L>0$ such that 
\begin{align*}
	\int_{-\hat L}^{\hat L}J_2(x_0-y)\rd y>B.
\end{align*}
By \eqref{3.13},
\begin{align*}
	&\int_\mathbb{R}J_2(x_n - y)v(t_n, y)\rd y\geq \int_{-\hat L}^{\hat L}J_2(x_n - y)v(t_n, y)\rd y\\
	= &  \int_{-\hat L}^{\hat L}J_2(x_n - y)\rd y+\int_{-\hat L}^{\hat L}J_2(x_n - y)[v(t_n, y)-1]\rd y\\
	\geq& \int_{-\hat L}^{\hat L}J_2(x_n - y)\rd y-\|J_2\|_\infty\int_{-\hat L}^{\hat L}|v(t_n, y)-1|\rd y\\
	\to & \int_{-\hat L}^{\hat L}J_2(x_0 - y)\rd y>B \ \  {\rm as}\ n\to \yy.
\end{align*}
This is a contradiction. Hence $B=1$.

Taking $(t,x)=(t_n,x_n)$ in \eqref{v-[a,b]} and letting $n\to\infty$, in view of $B=1$, we obtain 
\begin{align*}
	0=d_2-d_2\delta_2+\gamma \delta_2(1-\delta_2) = (1- \delta_2)(d_2 + \gamma\delta_2),
\end{align*}
which  implies $\delta_2=1$, as desired. The proof is finished.  
\end{proof}

Theorem \ref{th1.1}  clearly follows directly from Lemmas \ref{lem:2.2}, \ref{lemma3.5} and  \ref{lemma3.7}.

\section{Proof of Theorem \ref{th1.2}}  

Th arguments in the corresponding local diffusion case considered in \cite{DuLin} indicate that $h_\infty-g_\infty<\infty$ implies \eqref{3.4}.
Next we investigate whether this is also true in our nonlocal situation here. To this end we will make use of the estimates in Theorem \ref{th1.1} and regard \eqref{KnK1.2} as a perturbation of an ODE system.

Denote 
\begin{align*}
	m_1(t,x):=d_1\int_{g(t)}^{h(t)}J_{ 1}(x - y)u(t, y)\rd y\geq 0, \ \ m_2(t,x):=d_2\int_{\R}J_2(x - y)v(t, y)\rd y-d_2. 
\end{align*}
Then by Lemma \ref{lemma3.5} and Lemma \ref{lemma3.7}, it is easily seen that,   as $t\to \yy$,
\begin{equation}\label{m12}
\mbox{ $m_1(t,x)\to 0$ uniformly for $x\in\R$,\ $m_2(t,x)\to 0$ locally uniformly for $x\in\R$.}
\end{equation}
 The functions $u(t,x)$ and $v(t,x)$ satisfy, for $t\geq 0$, $x\in (g(t),h(t))$,
\[\begin{cases}
	u_t = m_1 + u(1 -d_1- u - kv),\\
	v_t = m_2 +d_2(1-v) + \gamma v(1 - v - hu),
\end{cases}
\]
which can be viewed as an ODE system for each fixed $x$.

 Recall that  under the assumption $h_\infty-g_\infty<\infty$, necessarily 
 \[
 \td d_1=d_1+k-1>0.\]
  Let
$F(s)$ and
$\Theta_1, \Theta_2$ be given by \eqref{F} and \eqref{theta}. 
Note that $F(0)=-\td d_1\gamma h<0$ and $F(1)=-d_2<0$. Therefore if $(\gamma, h, k, d_1, d_2)\in \Theta_1$, then
\begin{align}\label{F*}
	F_*:=\max_{s\in [0,1]} F(s)<0.
\end{align}

\begin{lemma}\label{Theta2} Assume $\tilde d_1=d_1+k-1>0$. Denote
	\begin{align*}
		a:=\gamma (1-hk), \ b:= \td d_1\gamma h-\gamma(1-hk)-d_2,\  c:=-\td d_1\gamma h,
	\end{align*}
	and so $F(s)=as^2+bs+c$.
	Then
	\begin{align*}
			(\gamma,h,k,d_1,d_2)\in \Theta_2 \Longleftrightarrow  a\leq c\ {\rm and}\  \sqrt{\frac ca}\leq \frac b {-2a}\leq 1.
	\end{align*}
\end{lemma}
\begin{proof}
Clearly $a+b+c+d_2=0$ and $c<0$. If $a=0$ then $F(s)=0$ implies $s=\frac{-c}{b}=\frac{-c}{-c-d_2}$ which can never be in [0,1].
If $a>0$ then $F(s)=0$ and $s\in [0,1]$ imply
\[ 
s=\frac{-b+\sqrt{b^2-4ac}}{2a}\leq 1.
\]
It follows that
\[ \sqrt{b^2-4ac}\leq 2a+b\implies b^2-4ac\leq b^2+4ab+4a^2\implies a(a+b+c)\geq 0\implies a+b+c\geq 0.
\]
But $a+b+c=-d_2<0$. We have thus proved that 
\[
(\gamma,h,k,d_1,d_2)\in \Theta_2\implies a<0.
\]
With $a<0$ and $b^2-4ac\geq 0$ we easily see that $\sqrt{b^2-4ac}< |b|$ and so 
\[
\frac{-b\pm\sqrt{b^2-4ac}}{2a} \mbox{ has the same sign as } \frac{-b}{2a}.
\]
Thus 
\[
(\gamma,h,k,d_1,d_2)\in \Theta_2\Longleftrightarrow  a<0<b, b^2-4ac\geq 0 \mbox{ and } \frac{-b+\sqrt{b^2-4ac}}{2a}\leq 1.
\]
With $a<0<b$, clearly $b^2-4ac\geq 0$ is equivalent to $b\geq 2\sqrt{ac}$, and $\frac{-b+\sqrt{b^2-4ac}}{2a}\leq 1$
is equivalent to $\sqrt{b^2-4ac}\geq 2a+b$, which holds trivially when $b\leq -2a$. If $b>-2a$ then 
\[\sqrt{b^2-4ac}\geq 2a+b \implies b^2-4ac\geq b^2+4ab+4a^2\implies a(a+b+c)\leq 0\implies a+b+c\geq 0
\]
which is impossible since $a+b+c=-d_2<0$. We have thus proved that
\[
(\gamma,h,k,d_1,d_2)\in \Theta_2\Longleftrightarrow  a<0 \mbox{ and } 2\sqrt{ac}\leq b\leq -2a.
\]
Finally, it is easily seen that
\[
a<0 \mbox{ and } 2\sqrt{ac}\leq b\leq -2a\Longleftrightarrow a\leq c\ {\rm and}\  \sqrt{\frac ca}\leq \frac b {-2a}\leq 1.
\]
This completes the proof.
\end{proof}
\begin{remark}\label{remak3} Suppose  $\tilde d_1>0$. It can be easily checked that $(\gamma,h,k,d_1,d_2)\in \Theta_1$ if
\begin{align*}
	d_1\geq 1 \  {\rm or}\ \dd kh\leq 1+\frac{d_2}{\gamma}.
\end{align*}
Indeed, 
\[d_1\geq 1\implies \tilde d_1\geq k\implies c\leq -\gamma kh=a-\gamma<a,
\]
and by Lemma \ref{Theta2} we see $(\gamma,h,k,d_1,d_2)\not\in \Theta_2$. Thus $d_1\geq 1$ implies $(\gamma,h,k,d_1,d_2)\in \Theta_1$.

If $kh\leq 1+\frac{d_2}{\gamma}$,  then  for $s\in (0,1)$,
\begin{align*}
	F(s)&=\gamma (1-hk)s^2-[\gamma(1-hk)+d_2]s+\td d_1\gamma h s-\td d_1\gamma h\\
	&< \gamma (1-hk)s^2-[\gamma(1-hk)+d_2]s\\
	&=\gamma (hk-1)(s-s^2)-d_2 s\leq d_2(s-s^2)-d_2 s=-d_2 s^2<0.
\end{align*} 
 Recalling $F(0)=-\td d_1\gamma h<0$, $F(1)=-d_1<0$, we see that
 $kh\leq 1+\frac{d_2}{\gamma}$ implies $(\gamma,h,k,d_1,d_2)\in \Theta_1$.
\end{remark}

\begin{lemma}\label{lem:2.10} Suppose $h_\yy-g_\yy<\yy$ and denote $\td v:=1-v$.  Then
there exists $M\in (0, 1)$ such that 
\begin{align}\label{3.40b}
	\limsup\limits_{t\to \yy} \max_{x\in [g_\yy,h_\yy]}u(t,x)\leq \max\{0, kM-\td d_1\},\ \ 	\limsup\limits_{t\to \yy} \max_{x\in [g_\yy,h_\yy]}|\td v(t,x)|\leq M.
\end{align}
If further $d_1\geq 1$, then \eqref{3.4}  holds.
\end{lemma}
\begin{proof}
From \eqref{3.9}, there is $m_*>0$ such that 
\begin{align*}
	\liminf\limits_{t\to \infty} \min_{x\in [g_\yy,h_\yy]}v(t,x)\geq m_*>0,
\end{align*}
which, combined with the fact that $\limsup\limits_{t\to \yy} \max_{x\in \R}v(t,x)\leq 1$, implies 
\begin{align*}
	\limsup\limits_{t\to \yy} \max_{x\in [g_\yy,h_\yy]}|\td v(t,x)|\leq 1-m_*.
\end{align*}
Then using the equation of $u_t$ and the uniformly convergence of $m_1(t,x)$ to $0$ in $\R$, we deduce by a simple comparison argument that
\begin{align*}
	\limsup\limits_{t\to \yy} \max_{x\in [g_\yy,h_\yy]}u(t,x)\leq \max\{0, k(1-m_*)-\td d_1\}. 
\end{align*}
We have thus proved \eqref{3.40b}.

Suppose $d_1\geq 1$, and so $\td d_1\geq k$. 
	In this case,  from $\td v=1-v\leq 1$, we deduce
	\begin{align*}
		u_t \leq m_1 + u(k-\td d_1- u )\leq m_1 -u^2 \mbox{ for } t>0,\  x\in [g(t),h(t)].
	\end{align*}
	For any given small $\epsilon>0$, since  $m_1(t,x)$ converges to $0$ as $t\to\infty$ uniformly for $x$ in $\R$, we can find $T=T_\epsilon>0$ so that
	\[ m_1(t,x)\leq \epsilon \mbox{ for } t\geq T,\ x\in\R.
	\]
	Let $U(t)$ be the solution of the ODE problem
	\[
	U'=\epsilon-U^2,\  U(T)=\|u(T,\cdot)\|_\infty.
	\]
	Then by the comparison principle we deduce
	$u(t,x)\leq U(t)$ for $t\geq T$, $x\in [g (t), h(t)]$. It follows that
	\[
	\limsup_{t\to\infty}\max_{x\in [g (t), h(t)]}u(t,x)\leq\lim_{t\to\infty} U(t)=\sqrt{\epsilon}.
	\]
	Letting $\epsilon\to 0$ we deduce
	\begin{align*}
	\lim_{t\to\yy } \max_{x\in [g(t),h(t)]} u(t,x)=0.
\end{align*}
Thus for any small $\delta>0$ there exists $T_\delta>0$ so that $u(t,x)\leq \delta$ for $t\geq T_\delta$ and all $x\in\R$. It follows that
\[
v_t \geq d_2\int_\mathbb{R}J_2(x - y)v(t, y)\rd y - d_2v + \gamma v(1 - v - h\delta) \mbox{ for } t\geq T_\delta,\ x\in\R. 
\]
Let $V(t,x)$ be the unique solution of
\[\begin{cases}
\dd V_t = d_2\int_\mathbb{R}J_2(x - y)V(t, y)\rd y - d_2V + \gamma V(1 - h\delta -V) \mbox{ for } t\geq T_\delta,\ x\in\R,\\[3mm]
V(T_\delta)=v(T_\delta, x) \mbox{ for } x\in\R.
\end{cases}
\]
Since $1-h\delta>0$ for all small $\delta>0$, it is well known that
\[
V(t,x)\to 1-h\delta \mbox{ as $t\to\infty$ locally uniformly for $x\in\R$.}
\]
By the comparison principle we have
$v(t,x)\geq V(t,x)$ for $t>T$ and $x\in\R$. It follows that
\[
\liminf_{t\to\infty} v(t,x)\geq 1-h\delta \mbox{ locally uniformly in } x\in\R.
\]
Letting $\delta\to 0$ we obtain
\[
\liminf_{t\to\infty} v(t,x)\geq 1 \mbox{ locally uniformly in } x\in\R.
\]
This implies, in view of \eqref{3.19}, 
	\begin{align}\label{3.35}
		\lim_{t\to\infty} v(t,x)= 1 \mbox{ locally uniformly in } x\in\R.
	\end{align}
	Hence \eqref{3.4} holds. The proof is complete.
	\end{proof}

\begin{lemma}\label{lem:Theta1} Suppose $h_\yy-g_\yy<\yy$.
	If 
	\[
	\mbox{$d_1<1$ and  $(\gamma,h,k,d_1,d_2)\in\Theta_1$,}
	\]
	 then  \eqref{3.4}  holds.
	\end{lemma}
		
\begin{proof} 
	Recall that $\td v=1-v$.  So the functions $u$ and $\td v$ satisfy
	\begin{equation}\label{u-tdv}
		\begin{cases}u_t = m_1 + u(-\td d_1- u + k\td v),\\
		\td v_t = -m_2+ \gamma hu -(d_2+\gamma)\td v +\gamma \td v(\td v - hu).
	\end{cases}
	\end{equation}	 

{\bf Claim 1}.  $\lim_{t\to\infty}\max_{x\in [g_\infty, h_\infty]}u(t,x)=0$.

 Suppose by way of contradiction that the desired conclusion in Claim 1 is not true. Then there exists a sequence $(t_n,x_n)$ with $t_n\to\infty$ and $x_n\in (g_\infty, h_\infty)$ such that
 \[
 \mbox{$u(t_n, x_n)> \epsilon_0>0$ for all $n\geq 1$  and some small $\epsilon_0>0$.}
 \]
  We may assume that $t_n> T_*$ for $n\geq 1$. For later arguments  we also assume that $\epsilon_0>0$ is small enough such that $\epsilon_0^2+\gamma h \epsilon_0+F_*<F_*/2<0$, where $F_*$ is given by \eqref{F*}.

 We will derive a contradiction by constructing a family  of  invariant sets for the solution pair $(u,v)$.  
 For
\begin{align*}
	\sigma\in (\frac{\td d_1}{k},1) \mbox{ and $\epsilon_0$ chosen above,}
\end{align*}
 define
 \[
 \epsilon(\sigma):=\min\{k\sigma-\td d_1, \epsilon_0\},\ M(\sigma):=k\sigma-\td d_1+ \epsilon(\sigma),
 \]
and
\begin{align*}
	A_\sigma:=\{(p,q)\in\R^2:0\leq p<M(\sigma) ,\  q< \sigma\}.
\end{align*} 
We will show that $A_\sigma$ is the desired invariant family.

Clearly $M(\sigma)$ is a continuous and strictly increasing function of $\sigma$ over the interval $[\frac{\td d_1}{k},1]$, with $M(\frac{\td d_1}{k})=0$.
 Let $\sigma_0\in (\frac{\td d_1}{k},1)$ be uniquely determined by
 \[
M(\sigma_0)=\epsilon_0.
 \]
 
 By \eqref{3.40b}, there exists $\sigma_*\in  (\sigma_0,1)$ and $T_*>0$ such that
 \[
 (u(t,x), \td v(t,x))\in A_{\sigma_*} \mbox{ for all } x\in [g_\infty, h_\infty],\ t\geq T_*.
 \]

In view of \eqref{m12},  by enlarging $T_*$ we may assume that 
\begin{align}\label{mi}
	&m_1(t,x)\leq \epsilon^2(\sigma_0),\ \ | m_2(t,x)|\leq \epsilon^2(\sigma_0)\ \mbox{ for } \ t\geq T_*,\ x\in [g_\yy,h_\yy].
\end{align}

For fixed $\sigma\in [\sigma_0, \sigma_*]$, $x\in [g_\infty, h_\infty]$,   $T\geq T_*$, $s>0$ and each $(p,q)\in A_\sigma$, we consider the solution map 
\[
S_x(T+s, T)(p,q):=(\bar p,\bar q),
\]
defined by $(\bar p, \bar q)=(u(T+s,x), \tilde v(T+s,x)$, where $(u(t,x),\td v(t,x))$ solves
\[
		\begin{cases}u_t = m_1 + u(-\td d_1- u + k\td v),& t>T,\\
		\td v_t = -m_2+ \gamma hu -(d_2+\gamma)\td v +\gamma \td v(\td v - hu),& t>T,\\
		u(T,x)=p, \ v(T,x)=q.
	\end{cases}
\]

{\bf Claim 2}. For each $\sigma\in [\sigma_0,\sigma_*]$, $t\geq T_*$ and $x\in [g_\yy,h_\yy]$,
\begin{align}\label{3.36}
S_x(t+s, t)(\partial A_{\sigma}) \subset   A_{\sigma}\ {\rm for \ all}\ s>0.	
\end{align}

To prove \eqref{3.36}, it suffices to show that
\begin{align}
	&u(t,x)=M(\sigma)\ {\rm and}\ \td v(t,x)\leq \sigma\ \implies  \ u_t(t,x)<\epsilon(\sigma)(-k\sigma+\td d_1)<0,\label{3.37}\\
	&u(t,x)\leq M(\sigma )\ {\rm and}\ \td v(t,x)= \sigma\ \implies \ \td v_t(t,x)<F_*/2<0.\label{3.38}
\end{align}
Indeed, by \eqref{u-tdv} and \eqref{mi}, $u(t,x)=M(\sigma)>0\ {\rm and}\ \td v(t,x)\leq \sigma$ imply $x\in (g (t), h(t))$ and
\begin{align*}
	u_t(t,x)\leq \epsilon^2(\sigma)-(k\sigma-\td d_1+\epsilon(\sigma))\epsilon(\sigma)=\epsilon(\sigma)(-k\sigma+\td d_1)<0,
\end{align*}
which proves \eqref{3.37}. To verify \eqref{3.38}, suppose $\td v(t, x) = \sigma$ and $u(t,x)\leq M(\sigma)$. Then by  \eqref{u-tdv} and \eqref{mi} we obtain
\begin{align*}
	\td v_t(t,x) =& -m_2+ \gamma h u(1-\td v) -(d_2+\gamma)\td v +\gamma \td v^2\\
	\leq&\ \epsilon^2(\sigma)+\gamma h (k\sigma-\td d_1+\epsilon(\sigma))(1-\sigma) -(d_2+\gamma)\sigma +\gamma  \sigma^2\\
	=&\ \epsilon^2(\sigma)+(1-\sigma)\gamma h \epsilon(\sigma)+\gamma (1-hk)\sigma^2+[\gamma(hk-1)-d_2+\td d_1\gamma h]\sigma-\td d_1\gamma h\\
	=&\ \epsilon^2(\sigma)+(1-\sigma)\gamma h \epsilon(\sigma)+F(\sigma)\leq \epsilon_0^2+\gamma h \epsilon_0+F_*<F_*/2<0.
\end{align*}
 Hence {\eqref{3.38}} holds.  This proves \eqref{3.36} and so
   Claim 2 holds true. 
 \medskip
 
 We are now ready to reach a contradiction and hence complete the proof of Claim 1.
 Consider $P_n(t):=(u(t,x_n), \td v(t,x_n))$ for $t\in [T_*, t_n]$. Since $u(t_n, x_n)> \epsilon_0$ and $t_n>T_*$, by Claim 2 and the definition of $\sigma_*$,  there exists $\sigma_n(t)\in [\sigma_0, \sigma_*]$ such that
 $P_n(t)\in \partial A_{\sigma_n(t)}$ and $\sigma_n(t)$ is nonincreasing in $t$ for $t\in [T_*, t_n]$.   In particular,
 \[
 P_n([T_*, t_n])\subset A:=A_{\sigma_*}\setminus \overline A_{\sigma_0}.
 \]
 If we define
 \[\begin{cases}
 A^+:=\{(p, q)\in A: p> M(q)\},\\ A^-:=\{(p, q)\in A: p< M(q)\},\\ \Gamma:=\{(p, q)\in A: p=M(q)\},
 \end{cases}
 \]
 then from \eqref{3.37} and \eqref{3.38} we see that, for $t\in [T_*, t_n]$,
 \[\begin{cases}
 P_n(t)\in A^+\cup\Gamma \implies \tilde v_t\leq -c_0,\\ P_n(t)\in A^-\cup\Gamma\implies u_t\leq -c_0,
 \end{cases}
 \]
 where $c_0:=\min\{F_*/2, \epsilon(\sigma_0)(k\sigma_0-\td d_1)\}>0$. 
 
 Define
 \[\begin{cases}
 I_n^+:=\{t\in (T_*, t_n): P_n(t)\in A^+\}, \\  I_n^-:=\{t\in (T_*, t_n): P_n(t)\in A^-\}, \\ I_n^0:=\{t\in (T_*, t_n): P_n(t)\in \Gamma\}.
 \end{cases}
 \]
 Then $I_n^+$ and $I_n^-$ are open sets (possibly empty for one of them). Hence each of them is the union of some (at most countably many) non-overlapping intervals when it is not the empty set. In the following, we derive a contradiction in each of the possible cases.
 
 If one of $I_n^+$ and $I_n^-$ is empty, say $I_n^-=\emptyset$, then 
 \[
 \dd -\sigma_*\leq \tilde v(t_n, x_n)-\tilde v(T_*, x_n)= \int_{T_*}^{t_n}\tilde v_t(t,x_n)dt\leq -c_0 (t_n-T_*)\to-\infty \mbox{ as } n\to\infty,
 \]
 which is a contradiction. Similarly $I_n^+=\emptyset$ leads to a contradiction.
 
 If both $I_n^+$ and $I_n^-$ are the union of some non-overlapping intervals, say for some non-empty but at most countable index sets $K_n^+$ and $K_n^-$,
 \[
 I_n^+=\cup_{k\in K_n^+}(s_k, t_k),\ I_n^-=\cup_{k\in K_n^-}(\tilde s_k, \tilde t_k),
 \]
 then 
 \[
 \begin{cases}k\in I_n^+ \mbox{ and } s_k\not=T_*\implies s_k\in \Gamma,\\  k\in I_n^- \mbox{ and } \tilde s_k\not=T_*\implies \tilde s_k\in \Gamma.
 \end{cases}
 \]
 By the invariance property of $A_\sigma$ we have
 \[\begin{cases}
 s_k\in \Gamma\implies \tilde v(t_k, x_n)\leq \tilde v(s_k, x_n), i.e., \dd \int_{s_k}^{t_k}\tilde v_t(t, x_n)dt\leq 0\\ \tilde s_k\in \Gamma \implies u(\tilde t_{k}, x_n)\leq u(\tilde s_k, x_n), i.e., \dd \int_{\tilde s_k}^{\tilde t_k} u_t(t, x_n)dt\leq 0,
 \end{cases}
 \]
 and 
 \[\begin{cases}
 s_k=T_*\implies \tilde v(t_k)\leq \sigma_*\implies \dd\int_{s_k}^{t_k}\tilde v_t(t, x_n)dt\leq \sigma_*,\\
  \tilde s_k=T_*\implies u(\tilde t_k)\leq M(\sigma_*)\implies \dd\int_{\tilde s_k}^{\tilde t_k}u_t(t, x_n)dt\leq M(\sigma_*).
  \end{cases}
 \]
 It follows that 
 \[
 \int_{I_n^+} u_t(t,x_n)dt\leq M(\sigma_*),\ \int_{I_n^-}\tilde v_t(t,x_n)dt\leq \sigma_*.
 \]
 Hence
 \[\begin{aligned}
 \dd &-\sigma_*\leq \tilde v(t_n, x_n)-\tilde v(T_*, x_n)\leq \sigma_*+\int_{I_n^+\cup I_n^0} \tilde v_t(t,x_n)dt\leq \sigma_*-c_0 |I_n^+\cup I_n^0|,\\
 \dd &-M(\sigma_*)\leq u(t_n, x_n)-u(T_*, x_n)\leq M(\sigma_*)+ \int_{I_n^-\cup I_n^0} u_t(t,x_n)dt\leq M(\sigma_*)-c_0 |I_n^-\cup I_n^0|.
 \end{aligned}
 \]
 Adding the above inequalities we obtain the following contradiction:
 \[
 -2[\sigma_*+M(\sigma_*)]\leq -c_0(|I_n^-\cup I_n^0|+|I_n^+\cup I_n^0|)\leq -c_0(t_n-T_*)\to -\infty \mbox{ as } n\to\infty.
 \]
  Claim 1 is thus proved.
 
\smallskip

{\bf Claim 3}. $\lim_{t\to\infty} v(t,x)=1$ locally uniformly in $x\in \R$.

This follows from Claim 1 in the same way as argued in the proof of Lemma \ref{lem:2.10} for the case $d_1\geq 1$. 
\end{proof}

\begin{lemma}\label{lem:Theta2} Suppose $h_\yy-g_\yy<\yy$.
		If 
		\[
		\mbox{$d_1<1$ and  $(\gamma,h,k,d_1,d_2)\in \Theta_2$,}
		\]
		 then either {\rm (i)} \eqref{3.4} holds or {\rm (ii)}  there is an open set $\Omega\subset (g_\yy,h_\yy)$ with
			$|\Omega|=h_\yy-g_\yy$, $\Omega\not=(g_\infty, h_\infty)$
			 such that 
		\begin{align}
		&\lim_{t\to \yy} (u(t,x),  v(t,x))=(0, 1)\  \mbox{ uniformly for $x$ in any compact subset of } \Omega,\label{3.33}\\
		&\lim_{t\to \yy} (u(t,x),  v(t,x))=(kx_*-\td d_1, 1-x_*)\  \mbox{ for } x\in  (g_\yy,h_\yy)\backslash \Omega,\label{3.34}
		\end{align}
	where $x_*>0$ is the smallest positive root of $F(s)=0$ in $[0, 1]$.
	\end{lemma}
	\begin{proof}
 We must have $kx_*-\td d_1>0$, since  $kx_*-\td d_1\leq 0$ implies
\begin{align*}
	0=F(x_*)= \gamma h (kx_*-\td d_1)(1-x_*) -(d_2+\gamma)x_* +\gamma x_*^2\leq -(d_2+\gamma)x_* +\gamma x_*^2<0,
\end{align*}
which is clearly impossible. 

{\bf Step 1}. We show that there exists an open set $\Omega$ as described in \eqref{3.33}. 

For
\begin{align*}
	\sigma\in (\frac{\td d_1}{k},x_*) \mbox{ and $\epsilon_1>0$ small to be determined as later argument desires,}
\end{align*}
 define
 \[
\tilde \epsilon(\sigma):=\min\{k\sigma-\td d_1, \epsilon_1\},\ \tilde M(\sigma):=k\sigma-\td d_1+ \tilde\epsilon(\sigma),
 \]
and
\begin{align*}
	\td A_\sigma:=\{(p,q)\in\R^2:0\leq p<\td M(\sigma) ,\  q< \sigma\}.
\end{align*} 
Fix $\td\sigma_*\in (\frac{\td d_1}k, x_*)$. Since $F(0)<0$ necessarily
\begin{align*}
\td F_*:=\max_{s\in [0,\td\sigma_*]} F(s)<0.
\end{align*}

From \eqref{3.5} and \eqref{3.13}, we see
\begin{align*}
\lim_{t\to\yy}\int_{g_\yy-1}^{h_\yy+1}[u(t,y)+|\td v(t,y)|] \rd y=0.
\end{align*}
Hence $u(t,x)+|\td v(t,x)|\to 0$ in measure for $x$ over $[g_\infty, h_\infty]$. By Egorov's theorem, for each small $\epsilon>0$, there exists a set $\Omega_\epsilon\subset (g_\infty, h_\infty)$ such that
\[
\begin{cases}
 |\Omega_\epsilon|\geq h_\infty-g_\infty-\epsilon,\\[1mm]
u(t,x)+|\td v(t,x)|\to 0 \mbox{ uniformly for } x\in \Omega_\epsilon \mbox{ as } t\to\infty.
\end{cases}
\]
Therefore, there exists $T_\epsilon>0$ such that 
\[
(u(t,x),\td v(t,x))\in \td A_{\tilde\sigma_*/2} \mbox{ for all  $t\geq T_\epsilon$ and $x\in\Omega_\epsilon$}.
\]
By the continuous dependence of $u(T_\epsilon, x)$ and $ \td v(T_\epsilon, x)$ on $x$, we see that there exists an open set  $ O_\epsilon$ such that
\[
\begin{cases} (g_\infty, h_\infty)\supset O_\epsilon\supset \Omega_\epsilon,\\[1mm]
(u(T_\epsilon,x),\td v(T_\epsilon,x))\in \td A_{\tilde\sigma_*} \mbox{ for all   $x\in\overline O_\epsilon$}.
\end{cases}
\]
We are now in a position to show
\begin{equation}\label{O}
\lim_{t\to\infty} \max_{x\in\overline O_\epsilon} u(t,x)=0
\end{equation}
by repeating the argument that leads to the conclusion in Claim 1 of the proof of Lemma \ref{lem:Theta1}, except that  we replace $(A_{\sigma_*}, [g_\infty, h_\infty], \sigma_*, F_*)$ there  by the above defined $(\td A_{\td \sigma_*}, \overline O_\epsilon, \tilde \sigma_*,\td F_*)$. 

Since $\epsilon>0$ is arbitrary, choosing $\epsilon_n>0$ converging to 0 monotonically as $n$ increasing to $\infty$,  we can obtain a sequence of open sets $\{O_{\epsilon_n}\}$
such that \eqref{O} holds for each $O_{\epsilon_n}$. Let $\Omega:=\cup_{n=1}^\infty O_{\epsilon_n}$. Then $\Omega\subset (g_\infty, h_\infty)$ is an open set with $|\Omega|=h_\infty-g_\infty$,
and since \eqref{O} holds for every $O_{\epsilon_n}$ we see that 
\begin{equation}\label{O-comp}
\lim_{t\to\infty}  u(t,x)=0 \mbox{ uniformly for $x$ in any compact subset of $\Omega$}.
\end{equation}
This implies that
\[
\td v_t = -\td m_2 -(d_2+\gamma)\td v +\gamma \td v^2
\]
with $\td m_2=\tilde m_2(t,x)\to 0$ as $t\to\infty$ uniformly for $x$ in any compact subset of $\Omega$. This fact, together with $\tilde v(t,x)\leq 1$ and
$\liminf_{t\to\infty} \td v(t,x)\geq 0$ uniformly in $x$, leads to
\[
\lim_{t\to\infty}\td v(t,x)=0 \mbox{ uniformly for $x$ in any compact subset of $\Omega$}
\]
by a simple comparison argument; we omit the details. Step 1 is now completed.

\medskip

{\bf Step 2.} Let $\Omega$ be the maximal open set contained in $(g_\infty, h_\infty)$ such that \eqref{3.33} holds. If $\Omega=(g_\infty, h_\infty)$, then \eqref{3.4} holds.

Arguing indirectly we assume that $\Omega=(g_\infty, h_\infty)$, but \eqref{3.4} does not hold. Since the first identity in \eqref{3.4} implies the second, we see that there must exist sequences $t_n\to\infty$ and $x_n\in (g (t_n), h(t_n))$ such that
\begin{equation}\label{epsilon0}
u(t_n, x_n)\geq\epsilon_0 \mbox{ for all } n\geq 1 \mbox{ and some } \epsilon_0>0.
\end{equation}

By passing to a subsequence we may assume that $x_n\to x^*\in [g_\infty, h_\infty]$. By our assumption, \eqref{3.33} holds and so
$(u(t,x^*),v(t, x^*))\to (0,1)$ as $t\to\infty$ if $x^*\in (g_\infty, h_\infty)$. If $x^*=g_\infty$ or $h_\infty$, then $u(t, x^*)\equiv 0$ and by \eqref{3.14}, $v(t,x^*)\to 1$ as $t\to\infty$. Thus we always have 
\[\mbox{$(u(t,x^*),v(t, x^*))\to (0,1)$ as $t\to\infty$. }
\]

Let $\td A_\sigma$, $\td\sigma_*$ and $\td F_*$ be defined as in Step 1 above. Then there exists $T_0>0$ such that 
\[
(u(T_0, x^*), \tilde v(T_0,x^*))\in \td A_{\td\sigma_*/2}.
\]
By continuous dependent of $u(T_0, x)$ and $\td v(T_0, x)$ on $x$, there exists a $\delta>0$ small such that 
\[
(u(T_0,x),\td v(T_0,x))\in \td A_{\tilde\sigma_*} \mbox{ for all   $x\in O_\delta:=[x^*-\delta, x^*+\delta]\cap [g_\infty, h_\infty]$}.
\]
This implies, as in Step 1,
\begin{equation}\label{O_delta}
\lim_{t\to\infty} \max_{x\in O_\delta} u(t,x)=0
\end{equation}
by repeating the argument that leads to the conclusion in Claim 1 of the proof of Lemma \ref{lem:Theta1}. But this is a contradiction to \eqref{epsilon0}. This completes Step 2.

{\bf Step 3}. 
Let $\Omega$ be the maximal open set contained in $(g_\infty, h_\infty)$ such that \eqref{3.33} holds.  Suppose that $\Omega\not=(g_\infty, h_\infty)$ and 
let $x_1\in (g_\infty, h_\infty)\setminus \Omega$. We are going to show that
\begin{align*}
	\lim_{t\to \yy} u(t,x_1)=kx_*-\td d_1,\ \ 	\lim_{t\to \yy} \td v(t,x_1)= x_*.
\end{align*}

Let
\begin{align*}
	v_1:=\liminf_{t\to\yy} \td v(t,x_1).
\end{align*}
Then there is a sequence $t_n\to\infty$ such that $ \td v(t_n,x_1)\to v_1$ and  $\td v_t(t_n,x_1)\to 0$ as $n\to \yy$. By passing to a subsequence we may also assume that $u(t_n,x_1)\to u_1$. From the equation of $\td v_t$ and $m_2\to 0$ as $t\to \yy$, we deduce 
\begin{align}\label{3.42}
	0=\gamma hu_1(1-v_1) -(d_2+\gamma) v_1 +\gamma  v_1^2.
\end{align}
We  show next that 
\begin{align}\label{3.43}
	v_1\geq x_*.
\end{align}
Arguing indirectly we assume $v_1<x_*$. Using $F(x_*)=0$ we see that the function 
\[
G(s):=\frac{(d_2+\gamma) s -\gamma  s^2}{\gamma h(1-s)}=\frac s{\gamma h}\Big(\frac {d_2}{1-s}+\gamma\Big)
\]
satisfies $G(x_*)=kx_*-\td d_1$. By \eqref{3.42} we obtain $G(v_1)=u_1$. Clearly $G(s)$ is strictly increasing for $s\in (0,1)$. Therefore 
\[
\mbox{$v_1<x_*$ implies
$u_1=G(v_1)<G(x_*)=k x_*-\td d_1$.}
\]
We show below that this leads to a contradiction and therefore \eqref{3.43} must hold. Indeed, $u_1<kx_*-\td d_1$ and $v_1<x_*$ imply that for any given $\hat\sigma_*<x_*$  close enough to $x_*$, with $\td A_\sigma$ as defined in Step 1,
\begin{align*}
	(u(t_n,x_1), \td v(t_n,x_1))\subset \td A_{\hat\sigma_*} \mbox{ for all large } n.
\end{align*}
Clearly $\hat F_*:=\max_{s\in[0,\hat\sigma_*]}F(s)<0.$
We may now repeat the argument that leads to the conclusion in Claim 1 of the proof of Lemma \ref{lem:Theta1} but with $(A_{\sigma_*}, [g_\infty, h_\infty], \sigma_*, F_*)$ there  replaced by $(\td A_{\hat \sigma_*}, x_1, \hat \sigma_*,\hat F_*)$, to conclude that
\[
\lim_{t\to\infty} u(t,x_1)=0.
\]
This implies $\lim_{t\to\infty} \td v(t,x_1)=0$ by making use of the equation satisfied by $\tilde v(t, x_1)$. But this is a contradiction to our assumption that $x_1\not\in\Omega$. We have thus proved \eqref{3.43}.

Let 
\begin{align*}
	v_2=\limsup_{t\to\yy} \td v(t,x_1).
\end{align*}
Then there exists a sequence $t_n$ such that 
\[
	v_2=\lim_{n\to\yy} \td v(t_n,x_1), \ \ 	0=\lim_{n\to\yy} \td v_t(t_n,x_1).
	\]
	By passing to a subsequence there exists $u_2\geq 0$ such that
	\[
		u_2:=\lim_{n\to\yy}  u(t_n,x_1).
		\]
		 Then the equation of $\td v_t$ and $m_2\to 0$ as $t\to \yy$ yield
 \begin{align}\label{v2}
 		0=\gamma hu_2(1-v_2) -(d_2+\gamma) v_2 +\gamma  v_2^2.
 \end{align}
By \eqref{3.40b}, we know 
\[
\mbox{$u_2\leq kM-\td d_1$ and $v_2\leq M$ for some $0<M<1$.}
\]

We now set to show that
\begin{align}\label{3.45}
	v_2\leq x_*.
\end{align}
Argue indirectly we assume that $v_2>x_*$ and seek a contradiction.  

Since $F(s)$ is a quadratic function satisfying $F(0)<0$ and $F(1)<0$, either $x_*$ is a degenerate root, namely
\begin{equation}\label{x*}
\mbox{ $F(s)<0$ for $s\in [0, x_*)\cup(x_*, 1]$,}
\end{equation}
 or there is another root $x^*\in (x_*, 1)$ such that
\begin{equation}\label{x**}
F(s)<0 \mbox{ for } s\in [0, x_*)\cup (x^*, 1],\ F(s)>0 \mbox{ for } s\in (x_*, x^*).
\end{equation}
Using the function $G (s)$ defined earlier we obtain from \eqref{v2} and $F(x_*)=0$ that
\[
u_2=G (v_2)>G (x_*)=kx_*-\td d_1.
\]

We will prove \eqref{3.45} by deriving a contradiction under the assumption $v_2>x_*$ for both cases \eqref{x*} and \eqref{x**}. 
\medskip

{\bf Claim 1}. Case \eqref{x*} leads to a contradiction.

When \eqref{x*} happens,  we fix
\begin{align*}
	\bar \sigma_*\in (M, 1) \mbox{  such that } k\bar\sigma_*-\td d_1>u_2,
\end{align*}
and  then fix
\begin{align*}
	\bar\sigma_0\in (x_*,\bar\sigma_*) \mbox{ close to $x_*$ such that $k\bar\sigma_0-\td d_1<u_2$}, \ \bar\sigma_0<v_2.
\end{align*}
It is now clear that
\begin{align*}
	 \bar F_*:=\max_{s\in[\bar\sigma_0, \bar\sigma_*]}F(s)<0.
\end{align*} 
We next choose $\epsilon_1>0$ small enough such that
\[
\epsilon_1^2+\gamma h \epsilon_1+\bar F_*<\bar F_*/2<0,\ k\bar\sigma_0-\td d_1+\epsilon_1<u_2,
\]
and define, for
$\sigma\in (x_*, \bar\sigma_*]$,
 \[
\bar\epsilon(\sigma):=\min\{\sigma-x_*, \epsilon_1\},\ \bar M(\sigma):=k\sigma-\td d_1+ \bar\epsilon(\sigma),
 \]
and
\begin{align*}
	\bar A_\sigma:=\{(p,q)\in\R^2: 0\leq p<\bar M(\sigma) ,\  q< \sigma\}.
\end{align*} 
Clearly $\bar M(\sigma)$ is continuous and strictly increasing in $\sigma$ with 
\[
\bar M(x_*)=kx_*-\td d_1<
\bar M(\bar\sigma_0)<u_2<\bar M(\bar\sigma_*).
\]
We also have
\[
\bar \sigma_0<v_2\leq M<\bar\sigma_*.
\]
Therefore, for all large $n$, 
\[
(u(t_n, x_1), \td v(t_n , x_1))\in \bar A_{\bar \sigma_*}\setminus \bar A_{\bar\sigma_0},
\]
and by \eqref{m12},
\begin{align}\label{mi-bar}
	&m_1(t,x)\leq \bar\epsilon(\bar \sigma_0)^2,\ \ | m_2(t,x)|\leq \bar\epsilon(\bar \sigma_0)^2\ \mbox{ for } \ t\geq t_n.
\end{align}
Fix such an $n$ and let the solution map $S_x(t+s, t)$ be defined as in Claim 1 of the proof of Lemma \ref{lem:Theta1}; then the same calculations as in the proof  there yield the following:

 For each $\sigma\in [\bar \sigma_0,\bar\sigma_*]$, $t\geq t_n$,
\begin{equation}\label{x1}\begin{cases}
S_{x_1}(t+s, t)(\partial \bar A_{\sigma}) \subset   \bar A_{\sigma}\ {\rm for \ all}\ s>0,\\
u(t,x_1)= \bar M(\sigma)\ {\rm and}\ \td v(t,x_1)\leq \sigma\ \implies  \ u_t(t,x_1)<\bar\epsilon(\sigma)(-k\sigma+\td d_1)<0,\\
u(t,x_1)\leq \bar M(\sigma)\ {\rm and}\ \td v(t,x_1)= \sigma\ \implies \ \td v_t(t,x_1)<\bar F_*/2<0.
\end{cases}
\end{equation}

Consider $P(t):=(u(t,x_1), \td v(t,x_1))$ for $t\in [t_n, t_{n+m}]$, $m\geq 1$. Since for every $m\geq 1$,
\[
(u(t_{n+m}, x_1), \td v(t_{n+m} , x_1))\in \bar A_{\bar \sigma_*}\setminus \bar A_{\bar\sigma_0},
\]
by \eqref{x1}  there exists $\sigma(t)\in [\bar\sigma_0, \bar\sigma_*]$ such that
 $P(t)\in \partial \bar A_{\sigma(t)}$ and $\sigma(t)$ is nonincreasing in $t$ for $t\in [ t_n, t_{n+m}]$. 
 
 We may now apply the same argument used to prove Claim 1 in the proof of Lemma \ref{lem:Theta1} to obtain  a contradiction,
 and Claim 1 is proved.
 
 \medskip
 
 {\bf Claim 2}. Case \eqref{x**} also leads to a contradiction.

 Since $u_2>kx_*-\td d_1$ and $ v_2>x_*$, we can find $\hat\sigma\in (x_*, x^*)$ such that
 \[
 u_2>k\hat\sigma-\td d_1,\ v_2>\hat\sigma.
 \]
 Fix $\hat\epsilon>0$ small so that 
 \[
 -\hat \epsilon^2-\gamma h \hat\epsilon+F(\hat\sigma)>0,\ \ k\hat\sigma-\td d_1- \hat\epsilon>0.
 \]
 Then choose $\hat T>0$ so that
 \[
 m_1(t,x)\leq \hat\epsilon^2,\ |m_2(t,x)|\leq \hat\epsilon^2 \mbox{ for } t\geq \hat T \mbox{ and } x\in [g_\infty, h_\infty].
 \]
 We now
define 
\begin{align*}
	\hat B:=\{(p,q): p> k\hat\sigma-\td d_1- \hat\epsilon,\  q> \hat\sigma\}.
\end{align*}
Clearly $(u_2, v_2)\in \hat B$ and therefore $(u(t_n, x_1), \tilde v(t_n, x_1))\in \hat B$ for all large $n$, say $n\geq n_0$. By enlarging $n_0$ we may also assume that
$t_{n_0}\geq\hat T$. By the continuous dependence of $u(t_{n_0}, x)$ and $\tilde v(t_{n_0},x)$ on $x$, there exists $\epsilon>0$ sufficiently small so that
\[
(u(t_{n_0}, x), \tilde v(t_{n_0}, x))\in \hat B \mbox{ for all } x\in [x_1-\epsilon, x_1+\epsilon].
\]
We show below that for each $x\in [x_1-\epsilon, x_1+\epsilon]$, the trajectory $\{(u(t,x), \td v(t,x)): t\geq t_{n_0}\}$ is trapped inside $\hat B$.
It suffices to show that
 for any $t\geq t_{n_0}$,
\begin{align}
	&u(t,x)=k\hat \sigma-\td d_1- \hat\epsilon\ {\rm and}\ \td v(t,x)\geq \hat\sigma\ {\rm implies } \ u_t(t,x)>0,\label{3.48}\\
	&u(t,x)\geq k\hat\sigma-\td d_1- \hat\epsilon\ {\rm and}\ \td v(t,x)= \hat\sigma\ {\rm implies } \ \td v_t(t,x)>0.\label{3.49}
\end{align}

Indeed,
 \eqref{3.48} follows from the simple calculation below
\[
u_t = m_1 + u(-\td d_1- u + k\td v)\geq (k\hat \sigma-\td d_1- \hat\epsilon)\hat\epsilon>0,
\]
and to verify \eqref{3.49}, we calculate
\begin{align*}
\td v_t(t,x) =& -m_2+ \gamma h u(1-\td v) -(d_2+\gamma)\td v +\gamma \td v^2\\
\geq &-\hat\epsilon^2+\gamma h (k\hat\sigma-\td d_1-\hat\epsilon)(1-\hat\sigma) -(d_2+\gamma)\hat\sigma +\gamma  \hat\sigma^2\\
=&-\hat\epsilon^2-(1-\hat\sigma)\gamma h \hat\epsilon+\gamma (1-hk)\hat\sigma^2+[\gamma(hk-1)-d_2+\td d_1\gamma h]\hat\sigma-\td d_1\gamma h\\
=&-\hat\epsilon^2-(1-\hat\sigma)\gamma h \hat\epsilon+F(\hat\sigma)>-\hat\epsilon^2-\gamma h \hat\epsilon+F(\hat\sigma)>0.
\end{align*}
Thus \eqref{3.49} holds and we have proved that
\[
(u(t,x), \td v(t,x))\in \hat B \mbox{ for all } t\geq t_{n_0},\ x\in [x_1-\epsilon, x_1+\epsilon].
\]
It follows that $[x_1-\epsilon, x_1+\epsilon]\subset [g_\infty, h_\infty]\setminus \Omega$, which is a contradiction to $|\Omega|=h_\infty-g_\infty$.
Claim 2 is now proved.

The above Claims 1 and 2 prove that \eqref{3.45} holds, namely  $v_2\leq x_*$. As we have already proved in \eqref{3.43} that $v_1\geq x_*$, by the definitions of $v_1$ and $v_2$ we must have $v_1=v_2=x_*$ and
\[
\mbox{$\tilde v(t,x_1)\to x_*$ as $t\to\infty$.}
\]
 It follows that $u(t,x_1)$ satisfies
\[
u_t=\tilde m_1+u(kx_*-\td d_1-u)
\]
with $\td m_1=\tilde m_1(t, x_1)\to 0$ as $t\to\infty$,  which implies
\[
u(t,x_1)\to kx_*-\td d_1 \mbox{ as } t\to\infty.
\]
This concludes Step 2 and the proof of the lemma is now complete.
\end{proof}

Theorem \ref{th1.2} now follows directly from Lemmas  \ref{lem:2.10}, \ref{lem:Theta1} and \ref{lem:Theta2}.

\section{Proof of Theorem \ref{th1.4}}
We divide the proof into three steps.

{\bf Step 1.} We show that  $h_\infty=-g_\infty=\infty$.

We only show that $h_\infty=\yy$ implies $g_\infty=-\infty$, as it can be shown similarly that $g_\infty=-\infty$ implies $h_\infty=\infty$.

Assume on the contrary that $h_\infty=\infty>-g_\infty$. Then by Remark \ref{remark3.6}, 
	\begin{align}\label{3.49a}
	\lim\limits_{t \rightarrow \infty}\int_{-\infty}^{L}u(t,y)\rd y = 0 \mbox{ for every } L>0.
\end{align}
Since we always have
\[
\limsup_{t\to\infty} v(t,x)\leq 1 \mbox{ uniformly for } x\in\R,
\]
for any given $\epsilon>0$ small, there exists a  large $T=T_\epsilon>0$ such that
\begin{align*}
		\dd	u_t \geq  d_1\int_{g (T)}^{h(T)}J_1(x - y)u(t, y)\rd y - d_1u + u[1-(k+\epsilon)-u] \mbox{ for } 
	t > T, ~g (T) < x < h(T).
\end{align*}

However,  as $1-k-\epsilon>0$ and we can make
   $h(T)-g(T)$ as large as we want by enlarging $T$,   the above inequality for $u$ implies, for such $\epsilon$ and $T$, by the comparison principle  and \cite[Propositions 3.5 and 3.6]{Cao-2019-JFA}, that 
\begin{align*}
	\liminf_{t\to \yy} \inf_{x\in [g(T),h(T)]} u(t,x)>0.
\end{align*}
This contradicts \eqref{3.49a}. Step 1 is finished.

{\bf Step 2.} We show that  $h\geq 1$ implies 
\begin{equation}\label{1-0}\mbox{
    $\lim\limits_{t \rightarrow \infty}u(t, x) = 1, \lim\limits_{t \rightarrow \infty}v(t, x) = 0$ locally uniformly for $x\in\R$.}
    \end{equation}

 We always have 
	\begin{equation}\label{leq1}
		\begin{aligned}
		\limsup\limits_{t\to\yy} [\sup_{x\in \R} u(t, x)]\leq 1,\ \ 
		\limsup\limits_{t\to\yy} [\sup_{x\in \R} v(t, x)]\leq 1.
		\end{aligned}
	\end{equation}
In view  $1 - k > 0$,  by making use of Lemma 3.14 in \cite{DY-2022-DCDS}, we derive
\begin{equation}
	\begin{aligned}
		\liminf\limits_{t\to\yy} u(t, x)\geq 1 - k : = \underline{u}_1~{\rm ~locally ~uniformly ~in}~ \mathbb{R}.
		\nonumber
	\end{aligned}
\end{equation}
The following proof will be presented according to two cases. 

{\bf Case 1}. $1 - h(1-k) \leq 0$.

In this case $1 - h\underline{u}_1 \leq 0.$ By  Proposition 4 (i) in \cite{Zwy-2022-dcds-b}, we have 
$$\limsup\limits_{t\to\yy} v(t, x)\leq 0 ~{\rm ~locally ~uniformly ~in}~ \mathbb{R}.$$
Since $v(t, x) \geq 0$, it follows that $\lim\limits_{t \rightarrow \infty}v(t, x) = 0$ locally uniformly in $\R$. 

Fix $L\gg 1$ and $0 < \varepsilon_1 \ll 1$. There  exists $T=T_{\varepsilon_1,L} > 0$ such that $v(t, x) \leq  \varepsilon_1$ and $h(T)-g(T)\gg L$ for $t > T$ and  $x\in [-L,L]$.   Thus, $u$ satisfies
\begin{align*}
	\dd	u_t \geq  d_1\int_{-L}^{L}J_1(x - y)u(t, y)\rd y - d_1u + u(1-k\epsilon_1 - u )  \ \mbox{ for } 
t > T, ~ -L < x < L.
\end{align*}
Since both $0<\epsilon_1\ll 1$ and $L\gg 1$ can be chosen arbitrarily, a simple comparison argument can be used to show that  $\liminf\limits_{t\to\yy} u(t, x)\geq 1$ locally uniformly for $x\in \R$. This, combined with $	\limsup\limits_{t\to\yy} u(t, x)\leq 1$,  gives $\lim\limits_{t \rightarrow \infty}u(t, x) = 1$ locally uniformly for $x\in \R$. Thus \eqref{1-0} holds in Case 1.

{\bf Case 2}. $1 - h(1 - k) > 0.$ 

Now
$$1 - h\underline{u}_1 =  1 - h(1-k) > 0.$$
By  Lemma 3.14 (ii) in \cite{DY-2022-DCDS}, we have  
$$\limsup\limits_{t\to\yy} v(t, x)\leq 1 - h\underline{u}_1 : = \overline{v}_2 ~{\rm ~locally ~uniformly ~in}~ \mathbb{R}.$$ 
Clearly 
$\underline{u}_2:=1 - k\overline{v}_2 = 1 - k(1- h(1 - k)) = (1-k)(1 + kh) > 0.$ According to  Lemma 3.14 (i) in \cite{DY-2022-DCDS}, we have
$$\liminf\limits_{t\to\yy} u(t, x)\geq \underline u_2  ~{\rm ~locally ~uniformly ~in}~ \mathbb{R}.$$ 

If $1-h\underline{u}_2 \leq 0,$  then similar to Case 1, we deduce $ \lim\limits_{t \rightarrow \infty}v(t, x) = 0$ and then $\lim\limits_{t \rightarrow \infty}u(t, x) = 1$ locally uniformly for $x\in\R$.

If $1-h\underline{u}_2 > 0,$ then
$$\limsup\limits_{t\to\yy} v(t, x)\leq 1 - h\underline{u}_2 : = \overline{v}_3 ~{\rm ~locally ~uniformly ~in}~ \mathbb{R}.$$ 
This and $\underline u_3:=1 - k\overline{v}_3 = (1 - k)(1 + kh + kh^2) > 0$ imply, as above,
$$\liminf\limits_{t\to\yy} u(t, x)\geq   \underline{u}_3 ~{\rm ~locally ~uniformly ~in}~ \mathbb{R}.$$ 

Continue with this procedure, we will obtain a sequence $\{\underline u_j\}$ and $\{\overline v_j\}$ such that
 \[
 \overline{v}_{j+1} = 1 - h\underline{u}_j, \ \underline{u}_{j + 1} = 1 - k\overline{v}_{j+1} \mbox{ for } j=1,2,...
 \]
 and there are two possibilities: 
 
  (a) there is a first $j\geq 1$ such that
    $h\underline{u}_j \geq 1$, then as in Case 1 we deduce \eqref{1-0}.
        
(b) $h\underline{u}_j < 1$ for all $j=1,2,...$. Then repeating the above analysis we obtain
$$\limsup\limits_{t\to\yy} v(t, x)\leq \overline{v}_j ~{\rm ~locally ~uniformly ~in}~ \mathbb{R},$$
$$\liminf\limits_{t\to\yy} u(t, x)\geq \overline{u}_j ~{\rm ~locally ~uniformly ~in}~ \mathbb{R}.$$
Moreover, $1\geq \underline{u}_j = (1-k)\sum_{i=0}^{j-1}(hk)^i $ for every $j\geq 1$. This implies that 
 $hk < 1$ and
$$\lim\limits_{j\rightarrow \infty}(\underline{u}_j, \overline v_j) = (\frac{1 - k}{1 - hk} , \frac{1 - h}{1 - hk}).$$
Since $\overline{v}_j > 0$ for all $j$ we further deduce $h \leq 1$. 

Summarising, we see that case (a) must happen when $h>1$. When $h=1$, if case (a) happens then the above discussion indicates that \eqref{1-0} holds, and if case (b) happens, then 
\[
\lim\limits_{j\rightarrow \infty}(\underline{u}_j, \overline v_j) = (\frac{1 - k}{1 - hk} , \frac{1 - h}{1 - hk})=(1,0),
\]
which, in view of \eqref{leq1},  again implies \eqref{1-0}. Therefore $h\geq 1$ always leads to \eqref{1-0}. This concludes Step 2.

{\bf Step 3.} We show that $h\in (0, 1)$ implies
 $$\lim\limits_{t\rightarrow\infty}u(t, x) = \frac{1 - k}{1 - hk}, ~\lim\limits_{t\rightarrow\infty}v(t, x) = \frac{1 - h}{1 - hk} {\rm ~locally~ uniformly~ in~} \mathbb{R}.$$
 
 In this situation, apart from the sequences
  $\{\underline u_j\}$ and $\{\overline v_j\}$ obtained in Step 2, we can use $0<h<1$ to define another two analogous sequences
    $\{\overline u_j\}$ and $\{\underline v_j\}$ with $\underline v_1=1-h$
  such that
 \[
 \overline{u}_{j+1} = 1 - k\underline{v}_j, \ \underline{v}_{j + 1} = 1 - h\overline{u}_{j+1} \mbox{ for } j=1,2,...,
 \]
 $$\limsup\limits_{t\to\yy} u(t, x)\leq \overline{u}_j ~{\rm ~locally ~uniformly ~in}~ \mathbb{R},$$
$$\liminf\limits_{t\to\yy} v(t, x)\geq \underline{v}_j ~{\rm ~locally ~uniformly ~in}~ \mathbb{R}.$$
It should be noted that $0<k<1$ and $0<h<1$ guarantee that these sequences are defined for all $j\geq 1$.

It follows that 
\[
\underline{v}_j = (1-h)\sum_{i=0}^{j-1}(hk)^i \to \frac{1-h}{1-hk} \mbox{ as } j\to\infty,
\]
and so $\overline u_j=1-k\underline v_j\to \frac{1-k}{1-hk} \mbox{ as } j\to\infty$.
We thus obtain 
$$\limsup\limits_{t\to\yy} u(t, x)\leq \frac{1 - k}{1 - hk}, ~\liminf\limits_{t\to\yy} v(t, x)\geq \frac{1 - h}{1 - hk} \mbox{ locally uniformly for } x\in\R.$$

From the sequences $\{\underline u_j\}$ and $\{\overline v_j\}$ obtained in Step 2,  we also have
$$\liminf\limits_{t\to\yy} u(t, x)\geq \frac{1 - k}{1 - hk}, ~\limsup\limits_{t\to\yy} v(t, x)\leq \frac{1 - h}{1 - hk} \mbox{ locally uniformly for } x\in\R.$$
The proof of Theorem \ref{th1.4} is now complete. \hfill $\Box$

\end{document}